\documentclass[onetabnum,final]{siamltex1213}
\usepackage[notcite,notref]{showkeys}
\usepackage{amsmath,mathrsfs,version}
\usepackage{amsfonts}
\usepackage{amssymb}
\usepackage{algorithm}
\usepackage{algorithmic}
\usepackage{bm}
\usepackage{color}
\usepackage{float}
\usepackage{graphicx}
\usepackage{epstopdf}
\usepackage{subfigure}

\title
{Rational spectral  methods for PDEs involving  fractional Laplacian  in unbounded domains\thanks{The work of the first author is supported by the National Natural Science Foundations of China under grant 91630312. The research of the second author is supported by  Singapore MOE AcRF Tier 2 Grants: MOE2017-T2-2-014 and MOE2018-T2-1-059.  The third author is supported by a Hong Kong PhD Fellowship. The last author is partially supported by the NSF of China (under grant numbers 11822111, 11688101, 91630203, 11571351, and 11731006), the science challenge project (No. TZ2018001), NCMIS, and the youth innovation promotion association (CAS).}}

\author
{
Tao Tang\footnotemark[1]
\and Li-Lian Wang\footnotemark[2]
\and Huifang Yuan\footnotemark[3]
\and Tao Zhou\footnotemark[4]
}

\newtheorem{remark}{Remark}[section]

\begin{document}
\maketitle

\renewcommand{\thefootnote}{\fnsymbol{footnote}}

\footnotetext[1]{Department of Mathematics, Southern University of Sciences and Technology, Shenzhen, China (tangt@sustc.edu.cn)}

\footnotetext[2]{Division of Mathematical Sciences, School of Physical and Mathematical Sciences, Nanyang Technological University, Singapore, 637371 (lilian@ntu.edu.sg)}

\footnotetext[3]{Department of Mathematics, Hong Kong Baptist University, Hong Kong, China (13480510@life.hkbu.edu.hk)}

\footnotetext[4]{LSEC, Institute of Computational Mathematics, Academy of Mathematics and Systems Science,
Chinese Academy of Sciences, Beijing, China (tzhou@lsec.cc.ac.cn)}

\renewcommand{\thefootnote}{\arabic{footnote}}

\slugger{mms}{xxxx}{xx}{x}{x--x}
\begin{abstract}
Many PDEs involving fractional Laplacian are naturally set in unbounded domains with underlying solutions decay very slowly, subject to certain power laws. Their numerical solutions are under-explored.
This paper aims at developing accurate spectral methods using rational basis (or modified mapped Gegenbauer functions) for such models in unbounded domains. The main building block of the spectral algorithms is the explicit representations for the Fourier transform and fractional Laplacian of the rational basis, derived from
some useful integral identites related to modified Bessel functions.    With these at our disposal, we
can construct rational spectral-Galerkin and direct collocation schemes by pre-computing  the associated fractional differentiation matrices. We obtain optimal error estimates of rational spectral approximation in the fractional Sobolev spaces, and analyze the optimal convergence of the proposed Galerkin scheme.
We also provide ample numerical results to show that the rational method outperforms the Hermite function approach.
\end{abstract}

\begin{keywords} Fractional Laplacian,   Gegenbauer polynomials, modified rational functions, unbounded domains, Fourier transforms, spectral methods.
\end{keywords}

\begin{AMS} 65N35, 65M70, 41A05, 41A25.
\end{AMS}

\pagestyle{myheadings}
\thispagestyle{plain}
\markboth{Tang et al}{Spectral methods  for fractional Laplacian}

\section{Introduction} Diffusion  is a ubiquitous physical process, typically modeled by
partial differential equations (PDEs) with usual Laplacian operators. Although they can describe
the anisotropy of diffusion,  many systems in science, economics and engineering
exhibit anomalous diffusion, which can be more accurately and realistically modelled by
PDEs with fractional Laplacian operators   \cite{physics_II,BFW98,dPQRV11}.
In the past decade,  tremendous research attention has  been paid to  the analysis and numerical studies
of fractional PDEs. The finite difference method and the finite element method are two widely studied methods in this direction (see, e.g., \cite{differenceHuang,differenceJi,elementJin,CX13,TD15,BJ17,WZ17,SX17,WZ18,GLW18,TYZ18} and references therein).  Most of efforts are devoted to dealing with the nonlocal nature or singularities of the fractional operators. Another powerful approach is the spectral method, which is more suitable for the non-local feature of the fractional operators (see, e.g.,  \cite{spectralgeorge,spectralshen,Chen_Mao_Shen,KZK16,LZK16,SS17,SS18,YM18,SLW18}).
However, most of these works  are  for fractional problems in bounded domains.  In particular, we refer to Bonito et al \cite{BBNO2018} for an up-to-date  review of the various numerical methods  for fractional diffusion based on different formulations of the fractional Laplacian.


 It is known that many  physically motivated fractional diffusive problems are naturally set in unbounded domains, but their  investigation is still under-explored. For usual PDEs in unbounded domains, several approaches have been widely used in practice (see, e.g.,  \cite{Boyd01,She.W09} and the original references cited therein). The first is direct domain truncation that works well for problems with rapidly decaying solutions, but  is not feasible for fractional problems as
the underlying solutions
usually decay slowly, subject to  certain power laws at infinity.
 On the other hand, the naive truncation introduces
nonphysical singularities at the interface where the unbounded domain is terminated.
The second is to design a suitable transparent boundary condition or  artificial sponge layer, but this appears highly nontrivial for the fractional Laplacian.  The third is the use of
orthogonal functions in unbounded domains, which has been successfully applied  to many usual PDEs (see, e.g., \cite{Tang93,Boyd01,FGT02,Ma.ST05,She.W09,SWT}).
Very recently, spectral methods for fractional PDEs on the half line are proposed by \cite{spectralArab,LZK16} -- using the generalized Laguerre functions as bases -- extending the idea of \cite{spectralgeorge}.  A two-domain spectral approximations by Laguerre functions is developed in \cite{CSW.18} for tempered fractional PDEs on the whole line.   Mao and Shen \cite{MaoShen} proposed both  spectral-Galerkin and collocation methods using Hermite functions for fractional PDEs in unbounded domains.  However, the collocation method therein relies on an equivalent formulation in frequency space by the Fourier transforms, and performs collocation methods to the equivalent formulation that involve forward/backward Hermite transforms.
Tang, Yuan and Zhou  \cite{hermitecollocation}  developed  \textit{direct} Hermite collocation methods with  explicit formulations for the differentiation matrices, which is therefore more robust for nonlinear problems.
Lastly, spectral approximation using non-classical orthogonal functions in unbounded domains -- image of classical Jacobi
polynomials through a suitable mapping, has proven to be more viable for usual PDEs
with solutions decaying algebraically
 (see, e.g.,  \cite{Boyd87a,Boyd01,GSW00,modifiedchebyshev,RYW19}), compared with approximation by Hermite/Laguerre functions.
 As such, the rational basis (or mapped Jacobi functions) should be more desirable for PDEs with fractional Laplacian, due to the slow decaying solution with long tails subject to certain power law. However, to the best of our knowledge, there is essentially no work available along this line.
 Moreover, the extension of the mapping technique  to the fractional setting is far from trivial, as we elaborate on below.

 In this paper, we intend to fill in this gap, and  design  rational spectral methods for a class of PDEs with
 fractional Laplacian in $\mathbb{R}^d$.
To fix the idea, we consider the model equation:
\begin{equation}\label{Laplace}
\begin{cases}
  (-\Delta)^{\alpha/2}u(x) +  \rho u(x) = f(x), \quad &x\in {\mathbb R}^d,\\
  u(x)=0,\quad &\lvert x\rvert\to \infty,
  \end{cases}
\end{equation}
for $\alpha\in (0,2),$ where the fractional Laplace operator is defined as in \cite{Lan72}:
\begin{equation}\label{sigular representation}
(-\Delta)^{\alpha/2}u(x):=C_{d,\alpha}{\rm p.v.}\int_{\mathbb{R}^d}\dfrac{u(x)-u(y)}{\lvert x-y\rvert^{d+\alpha}}dy \;\; \textmd{with} \;\;
C_{d,\alpha}:=\dfrac{{\alpha}2^{\alpha-1}\Gamma\Big(\dfrac{\alpha+d}{2}\Big)}{\pi^{d/2}\Gamma\Big(\dfrac{2-\alpha}{2}\Big)}.
\end{equation}
Here, p.v. stands for the Cauchy principal value, and $C_{d,\alpha}$ is a normalization constant.
Equivalently,  the fractional Laplacian can be defined as a pseudo-differential operator via the Fourier transform:
\begin{equation}\label{viafouriertransform}
  (-\Delta)^{\alpha/2}u(x):={\mathscr F}^{-1}\big[|\xi|^{\alpha} \mathscr{F}{\left[u\right]}({\xi})\big](x).
\end{equation}

For any expansion-based method, a critical issue is how to accurately evaluate  the point-wise value of the fractional Laplacian performing upon  the basis. For example, the key to the Hermite spectral method in \cite{MaoShen} is the use of the attractive  property that the Hermtie functions are the eigenfunctions of the Fourier transform, so the algorithm is largely implemented in the frequency $\xi$-space.
In contrast,  some analytically perspicuous
 formulas of $(-\Delta)^{\alpha/2}$ on the Hermite functions were derived in \cite{hermitecollocation}, which
led  themselves to the construction of efficient collocation algorithms in the physical $x$-space.
In the spirit of \cite{hermitecollocation}, we search for the analytic formulas for computing the fractional Laplacian of the rational basis functions -- the modified mapped Gegenbauer functions (MMGFs),  orthogonal with respect to a uniform weight. Although the formulas (see Theorem \ref{thm:mainformula}) are not as compact as those for the Hermite functions, we can accurately compute the fractional Laplacian of the rational basis  up to the degree $\thicksim 10^3$ by  using e.g., Maple or Mathematica. Moreover, with these analytic tools at our disposal, we are able to study their asymptotic behaviors and dependence of the
parameters so that the basis can be tailored to the decay rate of underlying solution.
We propose and analyze a spectral-Galerkin scheme, and obtain optimal estimates (see Theorem
\ref{Conv-thm}). We  also implement a direct collocation scheme based on the associated fractional differentiation matrices with the aid of the aforementioned explicit formulas. However, its error analysis  appears very  challenging and largely open. This is  mostly for the reason that  the fractional Laplacian takes
the rational basis to a  class of functions of completely different nature, as opposite to the usual Laplacian.
In the multi-dimensional case,  we implement the collocation schemes in the frequency space
(cf. \cite{MaoShen}), which relies on the approximability of  spectral expansions    to  ${\mathscr F}[f](\xi)/(|\xi|^{\alpha}+\rho).$ We show that the rational approach outperforms the Hermite method in accuracy. In fact,  it is common that
the Fourier transform of a functions decays much slower than the function itself, so the rational basis is more
desirable in this context.

The rest of the paper is organized as follows. In section \ref{sect:2},  we  collect some useful properties of  the Bessel  functions and Gegenbauer polynomials. In section \ref{sect:3}, we present the main formulas for computing the fractional Laplacian of the modified rational functions, and study the asymptotic properties.
In section \ref{sect:4},  we derive optimal error estimates of  the approximation by the modified rational functions in fractional Sobolev spaces.  We propose and analyse spectral-Galerkin methods using modified  rational basis functions in section \ref{sect:5}.  Then we implement the collocation methods in both one dimension and multiple dimensions in section \ref{sect:6}. The final section is for some concluding remarks.


%


%
%

\section{Preliminaries}\label{sect:2}
In this section, we  make necessary preparations for the algorithm development and analysis in the forthcoming sections.
More precisely, we review some relevant properties of the hypergeometric functions, Gegenbauer polynomials,  Bessel functions and  their interwoven relations.

\subsection{Bessel functions}
Recall that the Bessel function of the first kind of real order $\mu$ has the series expansion (cf. \cite{Olver2010Book}):
\begin{equation}\label{bessel1}
J_{\mu}(x)=\sum_{m=0}^\infty\, \frac{(-1)^m}{m! \,\Gamma(m+\mu+1)} \left(\frac{x}{2}\right)^{2m+\mu}.
\end{equation}
The modified Bessel functions of the first and second kinds are defined by
\begin{align}\label{mbessel}
&I_{\mu}(x)={\rm i}^{-\mu}J_{\mu}({\rm i} x), \quad 
K_{\mu}(x)=\frac{\pi}{2}\frac{I_{-\mu}(x)-I_{\mu}\left(x\right)}{\sin(\mu\pi)},
\end{align}
where ${\rm i}=\sqrt{-1}$ is the complex unit.
 For the modified Bessel functions of the second kind $K_{\mu}(x)$,  we have  the following important integral identities  (see \cite[P. 738]{tableofintegrals}):
 for  $-\lambda \pm \mu<1 $ and  $a,b>0,$
\begin{align}\label{Kintegral}
\int_{0}^{\infty} \!x^{\lambda}K_{\mu} ( ax)\cos(bx) dx &= {2^{\lambda-1}a^{-\lambda-1}\Gamma\Big(\frac{\mu+\lambda+1}{2}\Big)
\Gamma\Big(\frac{1+\lambda-\mu}{2}\Big)}\,\nonumber \\
&\;\;\;\;  \times  {}_{2}F_{1}\Big(\frac{\mu+\lambda+1}{2}, \frac{1+\lambda-\mu}{2};\!\frac{1}{2}; -\frac{b^2}{a^2} \Big),
\end{align}
and for $-\lambda \pm \mu<2$ and $a,b>0,$
\begin{align}\label{Kintegral2}
\int_{0}^{\infty}  x^{\lambda}K_{\mu} (ax) \sin(bx) dx & = \frac{{2^{\lambda}b\, \Gamma \big(\frac{2+\mu+\lambda}{2}\big) \Gamma\big(\frac{2+\lambda-\mu}{2}\big)}}{a^{2+\lambda}}\nonumber \\
& \;\;\;\; \times {}_{2}F_{1}\Big(\frac{2+\mu+\lambda}{2}, \frac{2+\lambda-\mu}{2}; \frac{3}{2} ;-\frac{b^2}{a^2}\Big).
\end{align}
Here, $\Gamma(\cdot)$ is the usual Gamma function, and ${}_{2}F_{1}$ is the hypergeometric function  defined in \eqref{F21defn} below.

\subsection{Hypergeometric functions}
For  any real $a,b,c$ with $c\not=0, -1,-2,\cdots,$  the hypergeometric function  is a power series defined by 
\begin{equation}\label{F21defn}
{}_{2}F_{1}\left(a,b;c;x\right)=\sum_{k=0}^{\infty}\frac{\left(a\right)_{k}\left(b\right)_{k}}{\left(c\right)_{k}}\frac{x^{k}}{k!}, \quad {\rm for}\;\;  |x|<1,
\end{equation}
 and by analytic continuation elsewhere (cf. \cite[P. 1014]{tableofintegrals} or \cite[Ch. 2]{Andrews99}).  Here $(a)_{k}$ is the rising Pochhammer symbol, i.e.,
$$(a)_0=1,\quad (a)_k=a(a+1)...(a+k-1)=\frac{\Gamma(a+k)}{\Gamma(a)}, \;\;\; k\in {\mathbb N}.$$
It is known that  the series ${}_2F_1(a,b;c; x)$ is absolutely convergent for all $|x|<1.$ Moreover,
 (i) if $c-a-b>0,$ the series ${}_2F_1(a,b;c; x)$  is absolutely convergent at $x=\pm 1;$  (ii)
if $-1<c-a-b\le 0,$ the series ${}_2F_1(a,b;c; x)$  is conditionally convergent at $x=-1,$ but it is divergent  at $x=1;$
(iii) if $c-a-b\le -1$, it diverges at $x=\pm 1.$
Its divergent behaviour at $x=1$ can be  characterised as follows  (cf. \cite[Ch. 2]{Andrews99}).
 \begin{itemize}
\item If $c=a+b,$ then
\begin{equation}\label{Nist15421}
\lim_{z\to 1^-}\frac{ {}_2F_1(a,b;a+b;z)}{-\ln (1-z)}=\frac{\Gamma(a+b)}{\Gamma(a)\Gamma(b)}.
\end{equation}
\item If $c<a+b,$ then
\begin{equation}\label{Nist15421cc}
\lim_{z\to 1^-}\frac{ {}_2F_1(a,b;c;z)}{(1-z)^{c-a-b}}=\frac{\Gamma(c)\Gamma(a+b-c)}{\Gamma(a)\Gamma(b)}.
\end{equation}
\end{itemize}
From the definition \eqref{F21defn}, we can easily obtain
\begin{equation*}
\frac{d^{k}}{dx^{k}}{}_{2}F_{1}(a,b;c;x)=\frac{(a)_{k}(b)_{k}}{(c)_{k}}{}_{2}F_{1}(a+k,b+k;c+k;x).
\end{equation*}
According to  \cite[P. 1019]{tableofintegrals}, there holds
\begin{align}\label{hyperrecur}
(2a-c-ax+bx){}_{2}F_{1}&(a,b;c;x)+(c-a){}_{2}F_{1}(a-1,b;c;x)\nonumber\\
&+a(x-1){}_{2}F_{1}(a+1,b;c;x)=0.
\end{align}
We also recall the property of hypergeometric functions related to transformations of variable (cf. \cite[P. 390]{Olver2010Book}):
\begin{equation}\label{lintransf}
{}_{2}F_{1}(a,b;c;x)=(1-x)^{-a}\, {}_{2}F_{1}\Big(a,c-b;c;\frac x {x-1}\Big),  
\end{equation}
and  the Pfaff's formula on the linear transformation (cf. \cite[(2.3.14)]{Andrews99}): for integer $n\ge 0,$ 
\begin{equation}\label{hypertransform}
{}_{2}F_{1}(-n,b;c;x)=\frac{(c-b)_{n}}{(c)_{n}}{}_{2}F_{1}(-n,b;b-c-n+1;1-x).
\end{equation}
Like \eqref{Kintegral}-\eqref{Kintegral2}, the following integral formulas (cf. \cite[P. 825]{tableofintegrals}) play a very important role in the  algorithm development:
for real $\mu>0$ and real $a,b,c>0,$
\begin{align}\label{Kfourier1}
\int_{0}^{\infty}\cos(\mu x){}_{2}F_{1}\Big(a,b;\frac{1}{2};-c^2 x^{2}\Big)dx=2^{-a-b+1}\pi c^{-a-b} \mu^{a+b-1}\frac{K_{a-b}({\mu}/{c})}{\Gamma(a)\Gamma(b)},
\end{align}
and for $a,b>{1}/{2},$
\begin{align}\label{Kfourier2}
\int_{0}^{\infty}x\sin(\mu x){}_{2}F_{1}\Big(a,b;\frac{3}{2};-c^2 x^{2}\Big)dx=2^{-a-b+1}\pi c^{-a-b} \mu^{a+b-2}\frac{K_{a-b}({\mu}/{c})}{\Gamma(a)
\Gamma(b)}.
\end{align}

\subsection{Gegenbauer polynomials}
Gegenbauer polynomials, 
denoted by  $C_{n}^{\lambda}(t),$  \; $ t\in I:=(-1,1)$ and $\lambda>-1/2$,  generalise Legendre and  Chebyshev polynomials.
They are defined by the three-term recurrence relation (cf. \cite[P. 1000]{tableofintegrals}):
\begin{equation}\label{gegenbauerrecurrence}
\begin{split}
&nC_{n}^{\lambda}(t)=2t(n+\lambda-1)C_{n-1}^{\lambda}(t)-(n+2\lambda-2)C_{n-2}^{\lambda}(t), \;\; n\geq 2,\\
&C_{0}^{\lambda}(t)=1,\quad C_{n}^{\lambda}(t)=2\lambda t.
\end{split}
\end{equation}
They are orthogonal  with respect to the weight function $\omega_\lambda(t)=(1-t^2)^{\lambda-1/2}$:
\begin{equation}\label{gegenor}
\int_{-1}^{1}C_{n}^{\lambda}(t)C_{m}^{\lambda}(t)\omega_\lambda(t)\, dt
=\gamma_{n}^{\lambda}\delta_{nm},\quad  \gamma_{n}^{\lambda}=\frac{\pi 2^{1-2\lambda}\Gamma(n+2\lambda)}{n!\left(n+\lambda\right)\Gamma^2(\lambda)}.
\end{equation}
The Gegenbauer polynomials can be defined by the hypergeometric functions (\cite[P. 1000]{tableofintegrals}):
\begin{equation}\label{GenHyper}
\begin{split}
&C_{2n}^{\lambda}(t)=\frac{(-1)^n}{(\lambda+n)B(\lambda,n+1)}\, {}_{2}F_{1}
\Big(\!-n,n+\lambda;\frac{1}{2};t^2\Big);\\[4pt]
&C_{2n+1}^{\lambda}(t)=\frac{(-1)^n 2t}{B(\lambda,n+1)}\, {}_{2}F_{1}\Big(\!-n,n+\lambda+1;\frac{3}{2};t^2\Big),
\end{split}
\end{equation}
where $B(\cdot,\cdot)$ is the Beta function satisfying  (cf. \cite[P. 918]{tableofintegrals}):
\begin{equation}\label{Betafun}
B(x,y)=\frac{\Gamma(x)\Gamma(y)} {\Gamma(x+y)}.
\end{equation}
Using the linear transformation  \eqref{hypertransform} and \eqref{GenHyper}-\eqref{Betafun},  we  have
\begin{equation}\label{GenHyper2}
\begin{split}
 & C_{2n}^{\lambda}(t)= a_n^{\lambda} \,\, {}_{2}F_{1}\Big(\!-n,n+\lambda;\lambda+\frac{1}{2};1-t^2\Big),\\[4pt]
 & C_{2n+1}^{\lambda}(t)= b_n^{\lambda} \, t \, {}_{2}F_{1}\Big(\!-n,n+\lambda+1;\lambda+\frac{1}{2};1-t^2\Big),
 \end{split}
\end{equation}
where
\begin{equation}\label{anbnlam}
 a_n^{\lambda}=\frac{\left(\lambda\right)_{n}}{\left(1\right)_{n}}\frac{\left(\lambda+\frac{1}{2}\right)_{n}}{\left(\frac{1}{2}\right)_{n}},\quad  b_n^{\lambda}=\frac{2\lambda\left(\lambda+1\right)_{n}}{\left(1\right)_{n}}
\frac{\left(\lambda+\frac{1}{2}\right)_{n}}{\left(\frac{3}{2}\right)_{n}}.
\end{equation}
%
\begin{remark}
Note that when $\lambda=0$,  we understand  the classical Chebyshev polynomials in the sense of   
\begin{equation*}
T_{n}\left(t\right)=\frac{n}{2}\lim_{\lambda\to 0} \frac{C_{n}^{\lambda}(t)}{\lambda},\quad n\geq 1.
\end{equation*}
Correspondingly, it follows from \eqref{GenHyper2} that
\begin{equation*}
\begin{split}
&T_{2n}(t)={}_{2}F_{1}\Big(\!\!-n,n;\frac{1}{2};1-t^2\Big); \quad T_{2n+1}(t)=t\,{}_{2}F_{1}\Big(\!\!-n,n+1;\frac{1}{2};1-t^2\Big).
\end{split}
\end{equation*}
Here,  we still denote $T_{n}(t):=C_{n}^{0}(t)$.
\end{remark}

\section{Fractional Laplacian of the modified mapped Gegenbauer functions}\label{sect:3}
\setcounter{equation}{0}
\setcounter{theorem}{0}

In this section, we introduce the rational basis functions through the Gegenbauer polynomials with a singular mapping. For convenience, we  term the resulting mapped basis as modified mapped Gegenbauer functions (MMGFs), which are different from the usual mapped Gegenbauer functions by absorbing the weight function in the basis.
We also present the explicit formulas for the evaluation of their fractional Laplacian, which plays an essential role in the spectral algorithms.

\subsection{The mapping and MMGFs} Consider the one-to-one mapping between
$t\in I=(-1,1)$ and $x\in \mathbb R=(-\infty,\infty)$ of the form:
\begin{equation}\label{algebraicmapping}
 x=\frac{t}{\sqrt{1-t^2}} \quad\; {\rm or} \quad\;  t=\frac{x}{\sqrt{1+x^2}},\quad t\in I,\;\; x\in {\mathbb R}.
\end{equation}
It is clear that
\begin{equation}\label{dxdt}
  1-t^2=\frac 1 {1+x^2},\quad \frac{dx}{dt}=\frac{1}{(1-t^2)^{3/2}}.
\end{equation}

 \begin{definition}\label{gendef} For $\lambda>-1/2,$ let  $C_n^\lambda(t), t\in I=(-1,1), $  be the Gegenbauer polynomial of degree $n$ as defined in \eqref{gegenbauerrecurrence}. We define the modified mapped Gegenbauer functions (MMGFs) as
 \begin{equation}\label{modifiedrational}
R_{n}^{\lambda}\left(x\right):=\big(1+x^2\big)^{-\frac{\lambda+1}{2}}C_{n}^{\lambda}\Big(\frac{x}{\sqrt{1+x^2}}\Big),\quad x\in{\mathbb R},
\end{equation}
or equivalently,
 \begin{equation}\label{modifiedrational2}
R_{n}^{\lambda}\left(x\right)=S(t) C_{n}^{\lambda}(t),\quad S(t):=\sqrt{\omega_\lambda(t) \frac{dx}{dt}}=(1-t^2)^{\frac{\lambda+1}2},
\end{equation}
where $x,t$ are associated with the mapping  \eqref{algebraicmapping}.
\end{definition}

One verifies readily from \eqref{gegenor} and \eqref{modifiedrational}-\eqref{modifiedrational2} that
\begin{equation}\label{orthss}
\int_{-\infty}^{\infty}R_{n}^{\lambda}(x)R_{m}^{\lambda}(x)\,dx=\gamma_{n}^{\lambda}\delta_{nm}.
\end{equation}
Thanks to  \eqref{gegenbauerrecurrence} and \eqref{modifiedrational}, the MMGFs  satisfy the three-term recurrence relation:
\begin{equation}\label{modifiedrationalrecurrence}
\begin{split}
&nR_{n}^{\lambda}(x)=\frac{2x}{\sqrt{1+x^2}}(n+\lambda-1)R_{n-1}^{\lambda}(x)-
(n+2\lambda-2)R_{n-2}^{\lambda}(x),\;\;\; n\geq 2;\\[2pt]
&R_{0}^{\lambda}(x)=\frac{1}{(1+x^2)^{\frac{\lambda+1}{2}}},\quad R_{1}^{\lambda}(x)= \frac{2\lambda x}{(1+x^2)^{1+\frac{\lambda}{2}}}.
\end{split}
\end{equation}
Moreover, we can show that
\begin{equation*}
\lim_{x\to\infty}\left(1+x^2\right)^{\frac{\lambda+1}{2}}R_{n}^{\lambda}(x)=\frac{\left(2\lambda\right)_{n}}{n!}, \quad \lim_{x\to-\infty}\left(1+x^2\right)^{\frac{\lambda+1}{2}}R_{n}^{\lambda}(x)=(-1)^{n}\frac{\left(2\lambda\right)_{n}}{n!}.
\end{equation*}
It is clear that by \eqref{F21defn},  \eqref{GenHyper2}  and \eqref{dxdt}, we have
\begin{align}\label{Rnlambda1}
R_{2n}^{\lambda}(x)&=\frac{a_n^\lambda}{(1+x^2)^{\frac{\lambda+1}{2}}}\,
{}_{2}F_{1}\Big(\!-n,n+\lambda;\lambda+\frac{1}{2};\frac{1}{1+x^2}\Big) \nonumber \\[2pt]
&=a_n^\lambda\, \sum_{k=0}^{n}\frac{\left(-n\right)_{k}\left(n+\lambda\right)_{k}}
{\left(\lambda+\frac{1}{2}\right)_{k}k!}\frac{1}{\left(1+x^2\right)^{k+\frac{\lambda+1}{2}}},
 \end{align}
 and
\begin{align}
\label{Rnlambda2}
R_{2n+1}^{\lambda}(x)&=\frac{b_n^\lambda}{\left(1+x^2\right)^{\frac{\lambda+1}{2}}}\frac{x}{\sqrt{1+x^2}}
\,{}_{2}F_{1}\Big(\!-n,n+\lambda+1;\lambda+\frac{1}{2};\frac{1}{1+x^2}\Big)\nonumber \\[2pt]
&=b_n^\lambda\,
\sum_{k=0}^{n}\frac{\left(-n\right)_{k}\left(n+\lambda+1\right)_{k}}
{\left(\lambda+\frac{1}{2}\right)_{k}k!}\frac{x}{\left(1+x^2\right)^{k+\frac{\lambda}{2}+1}}.
\end{align}
It is seen that the MMGFs are expressed in terms of
\begin{equation}\label{modkexp}
  \frac{1} {(1+x^2)^\gamma}\;\; \text{with}\;\;  \gamma=k+\frac{\lambda+1} 2 \quad
  {\rm or}\quad  \frac{x} {(1+x^2)^\gamma}\;\; \text{with}\;\;  \gamma=k+\frac{\lambda} 2+1.
\end{equation}

\subsection{Formulas for computing fractional Laplacian of MMGFs}


In view of \eqref{Rnlambda1}-\eqref{modkexp}, we first compute the fractional Laplacian of the simple functions in \eqref{modkexp}.
\begin{theorem}\label{Lapdelta} For real $s>0,$ we have that for any $\gamma>0,$
\begin{equation}\label{even}
(-\Delta)^{s}\Big\{\frac{1}{\left(1+x^2\right)^{\gamma}}\Big\}=
A_s^{\gamma}\, {}_{2}F_{1}\Big(s+\gamma,s+\frac{1}{2};\frac{1}{2};-x^2\Big),
\end{equation}
and for any $\gamma>1/2, $
\begin{equation}\label{odd}
(-\Delta)^{s}\Big\{\frac{x}{\left(1+x^2\right)^{\gamma}}\Big\}= (2s+1)A_s^\gamma\,
\, x \, {}_{2}F_{1}\Big(s+\gamma,s+\frac{3}{2}\, ;\frac{3}{2};-x^2\Big),
\end{equation}
 where the factor
\begin{equation}\label{constAsf}
A_s^\gamma:=  \frac{2^{2s}\Gamma(s+\gamma)\Gamma(s+\frac{1}{2})}  {\sqrt{\pi}\Gamma(\gamma)}.
\end{equation}
\end{theorem}
\begin{proof}  Recall the formula (cf. \cite[15.4.6]{Olver2010Book}):
\begin{equation}\label{hyperration}
{}_{2}F_{1}(b,a;a;z)=(1-z)^{-b}.
\end{equation}
Note that \eqref{hyperration} also holds for $z<1,$ with the analytic extension by the transformation formula  \eqref{lintransf}
(see \cite[9.130]{tableofintegrals}).
Thus, we have
\begin{equation}\label{vexpF21}
v(x):=\frac{1}{(1+x^2)^{\gamma}}={}_{2}F_{1}\Big(\gamma,\frac{1}{2};\frac{1}{2};-x^{2}\Big).
\end{equation}
Then using \eqref{Kfourier1} with $\mu=\xi, \alpha=\gamma$ and $\beta=1/2,$  we obtain that for $\xi>0,$
\begin{equation}\label{neweqnA}
\begin{split}
\hat{v}(\xi)&:={\mathscr F}[v](\xi)=
\frac{1}{\sqrt{2\pi}}\int_{-\infty}^{\infty}\frac{e^{-{\rm i} x\xi}}{\left(1+x^2\right)^{\gamma}}dx=\frac{2}{\sqrt{2\pi}}\int_{0}^{\infty}\frac{\cos\left(x\xi\right)}{\left(1+x^2\right)^{\gamma}}dx\\
&=\sqrt{\frac{2}{\pi}}\int_{0}^{\infty}\cos\left(x\xi\right){}_{2}F_{1}\Big(\gamma,\frac{1}{2};\frac{1}{2};-x^{2}\Big)dx\\
&=\sqrt{2\pi}2^{-\gamma+\frac{1}{2}}\xi^{\gamma-\frac{1}{2}}\dfrac{K_{\gamma-\frac{1}{2}}(\xi)}
{\Gamma(\gamma)\Gamma(\frac{1}{2})}
= \dfrac{2^{1-\gamma}} {\Gamma(\gamma)} \xi^{\gamma-\frac{1}{2}}{K_{\gamma-\frac{1}{2}}(\xi)}.
\end{split}
\end{equation}
Note that for $\xi<0,$ we have $\hat v(\xi)= \hat v(-\xi)$.  Thus, from
the definition \eqref{viafouriertransform} and \eqref{neweqnA}, we obtain
\begin{equation}\label{neweqnB}
\begin{split}
(-\Delta)^{s} v(x) &= {\mathscr F}^{-1}\big[|\xi|^{2s}{\mathscr F}[v](\xi) \big]
=\frac{1}{\sqrt{2\pi}}\int_{-\infty}^{\infty}e^{{\rm i} x\xi}|\xi|^{2s}\hat{v}\left(\xi\right) d\xi\\
&=\frac{2^{1-\gamma}}{\sqrt{2\pi}\Gamma\left(\gamma\right)}\int_{0}^{\infty}\cos\left(x\xi\right)\xi^{2s+\gamma-\frac{1}{2}}
K_{\gamma-\frac{1}{2}}(\xi)\,d\xi.
\end{split}
\end{equation}
Then using the formula \eqref{Kintegral} with $\lambda=2s+\gamma-1/2, \mu=\gamma-1/2$ and $b=x,$   we find
\begin{equation}\label{neweqnC}
\begin{split}
(-\Delta)^{s} v(x)&=\frac{2^{2s}\Gamma(s+\gamma)\Gamma(s+\frac{1}{2})}
{\sqrt{\pi}\Gamma(\gamma)}\, {}_{2}F_{1}\Big(s+\gamma,s+\frac{1}{2};\frac{1}{2};-x^2\Big).
\end{split}
\end{equation}
Hence,  we derive  \eqref{even}. 

The formula \eqref{odd} can be derived in a similar  fashion.  Like \eqref{neweqnA}, we obtain from
 \eqref{Kfourier2} with $\mu=\xi, \alpha=\gamma$ and $\beta=3/2$ that for $\gamma>1/2$ and  $\xi>0,$
\begin{equation}\label{neweqnA0}
\begin{split}
{\mathscr F}[xv](\xi) & =
\frac{1}{\sqrt{2\pi}}\int_{-\infty}^{\infty}\frac{x e^{-{\rm i} x\xi}}{\left(1+x^2\right)^{\gamma}}dx=-\frac{2 {\rm i}}{\sqrt{2\pi}}\int_{0}^{\infty}\frac{x \sin(x\xi)}{\left(1+x^2\right)^{\gamma}}dx\\
&=- {\rm i} \sqrt{\frac{2}{\pi}}\int_{0}^{\infty} x \sin(x\xi)\, {}_{2}F_{1}\Big(\gamma,\frac{3}{2};\frac{3}{2};-x^{2}\Big) dx\\
&=-{\rm i} \sqrt{2\pi}2^{-\gamma-\frac{1}{2}}\xi^{\gamma-\frac{1}{2}}\dfrac{K_{\gamma-\frac{3}{2}}(\xi)}
{\Gamma(\gamma)\Gamma(\frac{3}{2})}
=-{\rm i}  \dfrac{2^{1-\gamma}} {\Gamma(\gamma)} \xi^{\gamma-\frac{1}{2}}{K_{\gamma-\frac{3}{2}}(\xi)}.
\end{split}
\end{equation}
Note that in this case, ${\mathscr F}[xv](-\xi)=-{\mathscr F}[xv](\xi).$ Similar to  \eqref{neweqnB}, we find
\begin{equation}\label{neweqnB0}
\begin{split}
(-\Delta)^{s} \big\{xv(x)\big\} &= {\mathscr F}^{-1}\big[|\xi|^{2s}{\mathscr F}[xv](\xi) \big]
=\frac{1}{\sqrt{2\pi}}\int_{-\infty}^{\infty}e^{{\rm i} x\xi}|\xi|^{2s}{\mathscr F}[xv](\xi) d\xi\\
&=\frac{2^{2-\gamma}}{\sqrt{2\pi}\Gamma (\gamma)}\int_{0}^{\infty}\sin(x\xi)\xi^{2s+\gamma-\frac{1}{2}}
K_{\gamma-\frac{3}{2}}\left(\xi\right)d\xi.
\end{split}
\end{equation}
 Thus, we derive from  \eqref{Kintegral2} with $\lambda=2s+\gamma-1/2, \mu=\gamma-3/2$ and $b=x$ that
\begin{equation}\label{neweqnC0}
\begin{split}
(-\Delta)^{s} \big\{xv(x)\big\} &=\frac{2^{2s+1}\Gamma(s+\gamma)\Gamma(s+3/2)}
{\sqrt{\pi}\Gamma(\gamma)}\, x \, {}_{2}F_{1}\Big(s+\gamma,s+\frac{3}{2}\, ;\frac{3}{2};-x^2\Big).
\end{split}
\end{equation}
Finally, the formula \eqref{odd} follows from the property: $\Gamma(z+1)=z\Gamma(z).$
\end{proof}

Using the transformation formula \eqref{lintransf}, we can represent the formulas in Theorem \ref{Lapdelta} in the terms of the hypergeometric function defined through the series in \eqref{F21defn}. Note that the former  is more convenient for computation, while the latter is more suitable for analysis. 
\begin{corollary}\label{Lapdelta-col} For real $s>0,$ we have that for any $\gamma>0,$
\begin{equation}\label{even-col}
(-\Delta)^{s}\Big\{\frac{1}{\left(1+x^2\right)^{\gamma}}\Big\}=
\frac{A_s^{\gamma}} {(1+x^2)^{s+\gamma}}\,
 {}_{2}F_{1}\Big(\!-s, s+\gamma;\frac{1}{2};\frac{x^2}{1+x^2}\Big),
\end{equation}
and for any $\gamma>1/2, $
\begin{equation}\label{odd-col}
(-\Delta)^{s}\Big\{\frac{x}{\left(1+x^2\right)^{\gamma}}\Big\}= (2s+1)A_s^\gamma\,
\frac{x} {(1+x^2)^{s+\gamma}}\,
 {}_{2}F_{1}\Big(\!-s, s+\gamma;\frac{3}{2};\frac{x^2}{1+x^2}\Big),
\end{equation}
 where $A_s^\gamma $ is defined as in \eqref{constAsf}.
\end{corollary}

\smallskip

\begin{remark}\label{sint-rem}
It is seen that if $s$ is a positive integer, then the hypergeometric functions in \eqref{even-col} and
\eqref{odd-col} become finite series.  We can directly verify by using   \eqref{F21defn} that
\begin{equation}\label{Dsdecay}
(-\Delta)^{s}\Big\{\frac{1}{\left(1+x^2\right)^{\gamma}}\Big\}\sim \frac{1} {(1+x^2)^{s+\gamma}},\quad
(-\Delta)^{s}\Big\{\frac{x}{\left(1+x^2\right)^{\gamma}}\Big\}\sim \frac{x} {(1+x^2)^{s+\gamma}},
\end{equation}
for $s=1,2,\cdots,$ and $|x|\to \infty.$ However, for non-integer $s>0,$  the hypergeometric functions may diverge as $|x|\to\infty.$
Indeed, we find from \eqref{Nist15421} and \eqref{Nist15421cc} that
\begin{itemize}
\item[(i)] if $\gamma=1/2,$ then
\begin{equation}\label{neweqnAln}
(-\Delta)^{s}\Big\{\frac{1}{\sqrt{1+x^2}}\Big\}\sim \frac{\ln (1+x^2)} {(1+x^2)^{s+1/2}},
\end{equation}
\item[(ii)] if $\gamma>1/2,$ then
\begin{equation}\label{neweqnBln}
(-\Delta)^{s}\Big\{\frac{1}{(1+x^2)^{\gamma}}\Big\}\sim \frac{1} {(1+x^2)^{s+1/2}},
\end{equation}
\item[(iii)] if $0<\gamma<1/2,$ it has the  same behaviour  as in \eqref{Dsdecay}.
\end{itemize}

\smallskip
Similarly, we can analyse the behaviour  at infinity for \eqref{odd-col} for three cases: (i) $\gamma=3/2;$ (ii) $\gamma>3/2$ and
(iii) $1/2<\gamma<3/2.$
\end{remark}

\smallskip
\begin{remark}\label{newsA}
In a distinctive difference with the integer case, we see that the decay rate in the fractional case in \eqref{neweqnAln} is independent of $\gamma,$ if $\gamma>1/2.$
\end{remark}

%


With the above preparations, we can now derive the explicit representation of  the fractional Laplacian of  $\{R_n^\lambda\}$.
\begin{theorem}\label{thm:mainformula} For real $s>0$ and $\lambda>-1/2,$  the fractional Laplacian of the MMGFs can be represented by
\begin{equation}\label{LapR2n}
\begin{split}
 (-\Delta)^s & R_{2n}^\lambda(x)= 
   a_n^\lambda  \\ &
 \times  \sum_{k=0}^{n}\frac{(-n)_{k}(n+\lambda)_{k}}
{(\lambda+\frac{1}{2})_{k}\, k!}  A_s^{k+\frac{\lambda+1}{2}} {}_{2}F_{1}\Big(s+k+\frac{\lambda+1}{2},s+\frac{1}{2};\frac{1}{2};-x^2\Big),
 \end{split}
\end{equation}
and
\begin{equation}\label{LapR2np1}
\begin{split}
 (-\Delta)^s &R_{2n+1}^\lambda(x)=  (2s+1)\,b_n^\lambda\, x  
 \\ &  \times \sum_{k=0}^{n}\frac{(-n)_{k}(n+\lambda+1)_{k}}
{(\lambda+\frac{1}{2})_{k}k!} \, A_s^{k+\frac{\lambda}{2}+1}
 {}_{2}F_{1}\Big(s+k+\frac{\lambda}{2}+1,s+\frac 3 2;\frac{3}{2};-x^2\Big),
 \end{split}
\end{equation}
where the constants $a_n^\lambda, b_n^\lambda$ and $A_s^\gamma$ are the same as in \eqref{anbnlam} and
\eqref{constAsf}.
\end{theorem}
\begin{proof}  By \eqref{Rnlambda1},  we have
\begin{equation}\label{addA1}
\begin{split}
&(-\Delta)^s R_{2n}^\lambda(x)  =a_n^\lambda\, \sum_{k=0}^{n}\frac{(-n)_{k}(n+\lambda)_{k}}
{(\lambda+\frac{1}{2})_{k}k!} (-\Delta)^s \bigg\{\frac{1}{\left(1+x^2\right)^{k+\frac{\lambda+1}{2}}}\bigg\},
\end{split}
\end{equation}
so substituting \eqref{even} with $\gamma=k+\frac{\lambda+1}{2}$ into the above leads to \eqref{LapR2n}.

Similarly, we derive  from \eqref{Rnlambda2}  that
\begin{equation}\label{addA2}
\begin{split}
&(-\Delta)^s R_{2n+1}^\lambda(x)  =b_n^\lambda\,
\sum_{k=0}^{n}\frac{(-n)_{k}(n+\lambda+1)_{k}}
{(\lambda+\frac{1}{2})_{k}k!} (-\Delta)^s \bigg\{\frac{x}{\left(1+x^2\right)^{k+\frac{\lambda}{2}+1}}\bigg\},
\end{split}
\end{equation}
so substituting \eqref{odd} with $\gamma=k+\frac{\lambda}{2}+1$ into the above  leads to \eqref{LapR2n}.
\end{proof}
In view of  the asymptotic results  in Remark \ref{sint-rem},  we can analyse the decay rate of fractional Laplacian of the basis. Indeed, by  virtual of
\eqref{Dsdecay}-\eqref{neweqnBln}, we obtain from
  \eqref{addA1}-\eqref{addA2} that  (i) if $-1/2<\lambda<0,$ then
  \begin{equation}\label{lmda12}
   (-\Delta)^s R_{2n}^\lambda(x) \sim \frac{1} {(1+x^2)^{s+\frac{\lambda+1}2}};
   \end{equation}
(ii) if $\lambda=0,$  we have
  \begin{equation}\label{lmda13}
   (-\Delta)^s R_{2n}^\lambda(x) \sim \frac{\ln(1+x^2)} {(1+x^2)^{s+ 1/2}};
   \end{equation}
(iii) if $\lambda>0,$  we have
  \begin{equation}\label{lmda14}
   (-\Delta)^s R_{2n}^\lambda(x) \sim \frac{1} {(1+x^2)^{s+ 1/2}}.
   \end{equation}
   Similar results are available for $   (-\Delta)^s R_{2n+1}^\lambda(x).$ It is  noteworthy
   from \eqref{modifiedrational} and the above that the fractional Laplacian  on the basis does not always lead to the gain in decay rate of
   $1/(1+x^2)^s.$

\begin{remark} It is important to point out that the involved hypergeometric functions in \eqref{LapR2n} and  \eqref{LapR2np1} can be evaluated recursively by using \eqref{hyperrecur}.  Denote
$$
F_k(x)={}_2F_{1}(a,b;c;-x^2),\quad a=s+k+\frac{\lambda+1}{2},\;\; b=s+\frac{1}{2}, \;\; c=\frac{1}{2}.
$$
Then by \eqref{hyperrecur},  we have
\begin{equation}\label{F21recurr}
\begin{split}
F_{k+1}(x) &=\frac{c-a}{a(1+x^2)}\,F_{k-1}(x)+\frac{(2a-c)+(a-b)x^2} {a(1+x^2)} F_k(x),
  \end{split}
\end{equation}
for $k\ge 1.$   Similarly, we can efficiently compute the hyergeometric functions in \eqref{LapR2np1}.
\end{remark}

\begin{remark} To enhance the resolution of the basis, one can  also introduce  a scaling parameter $\mu>0$ (cf. \cite{SWT}). More precisely,  the algebraic mapping in \eqref{algebraicmapping} turns to
\begin{equation*}
x=\frac{\mu\,t}{\sqrt{1-t^2}},\quad t=\frac{x}{\sqrt{\mu^2+x^2}}.
\end{equation*}
The corresponding modified rational function can be defined as
\begin{equation*}
R_{n,\mu}^{\lambda}(x):=\frac{\mu^{\lambda+\frac{1}{2}}}{\left(\mu^2+x^2\right)^{\frac{\lambda+1}{2}}}C_{n}^{\lambda}\Big(\frac{x}{\sqrt{\mu^2+x^2}}\Big)=\mu^{-\frac{1}{2}}R_{n}^{\lambda}\Big(\frac{x}{\mu}\Big).
\end{equation*}
In fact, it is straightforward to extend the previous properties and formulas to the scaled basis. For simplicity, we omit the details.
\end{remark}

\section{Estimates of MMGF approximation in fractional Sobolev spaces}\label{sect:4}
In this section,  we analyse the approximation property by the modified rational basis functions in fractional Sobolev spaces.  We remark that there exist very limited results on the Legendre or Chebyshev rational approximations (see
\cite{modifiedchebyshev,modifiedlegendre,SWY14}). However, most of them are suboptimal.  Here, we derive the optimal estimates in more general settings.
\subsection{Fractional Sobolev spaces}
For real $r\ge 0,$ we define  the fractional Sobolev space (as in
\cite[P. 30]{Lions1972Book} and \cite[Ch. 1]{Agranovich2015Book}):
\begin{align}\label{Hssps}
H^r(\mathbb{R})=\Big\{u\in L^{2}\left(\mathbb{R}\right): \int_{\mathbb R} (1+\lvert\xi\rvert^{2})^r
\big|\mathscr{F}[u](\xi)\big|^{2}d\xi<+\infty\Big\},
\end{align}
equipped with the norm
\begin{align}\label{normHsp}
\big\lVert u\big\rVert_{H^r(\mathbb{R})}= \Big(\int_{\mathbb R} (1+\lvert\xi\rvert^{2})^r
\big|\mathscr{F}[u](\xi)\big|^{2}d\xi\Big)^{1/2}.
\end{align}

We have the following space interpolation property  (cf. \cite[Ch. 1]{Agranovich2015Book}).
\begin{lemma}\label{interpola}
For real $r_0,r_1\ge 0,$ let $r=(1-\theta) r_0+\theta r_1$ with $\theta \in[0,1]$. Then  for any $u\in H^{r_0}(\mathbb{R})\cap H^{r_1}(\mathbb{R}),$  we have
\begin{equation}\label{Hrinterp}
\|u\|_{H^{r}(\mathbb{R})}\le \|u\|^{1-\theta}_{H^{r_0}(\mathbb{R})} \, \|u\|^{\theta}_{H^{r_1}(\mathbb{R})},
\end{equation}
\end{lemma}
\begin{proof} For  the readers' reference,  we sketch the derivation of this interpolation property.  It is clear that by \eqref{normHsp},
\begin{equation*}\begin{split}
\|u\|^2_{H^{r}(\mathbb{R})}&=\int_{\mathbb{R}}(1+|\xi|^2)^{(1-\theta)r_0+\theta r_1}|\mathscr{F}[u](\xi)|^2d\xi\\
&=\int_{\mathbb{R}}\Big\{(1+|\xi|^2)^{(1-\theta)r_0}|\mathscr{F}[u](\xi)|^{2(1-\theta)}\Big\}\Big\{(1+|\xi|^2)^{\theta r_1}|\mathscr{F}[u](\xi)|^{2\theta}\Big\}d\xi.
\end{split}\end{equation*}
Using the H\"{o}lder's inequality with $p=1/(1-\theta)$ and $q=1/\theta$, we obtain
\begin{equation*}\begin{split}
\|u\|^2_{H^{r}(\mathbb{R})}
&\le
\Big\{\int_{\mathbb{R}}(1+|\xi|^2)^{r_0}|\mathscr{F}[u](\xi)|^{2}d\xi\Big\}^{1-\theta}
\Big\{\int_{\mathbb{R}}(1+|\xi|^2)^{ r_1}|\mathscr{F}[u](\xi)|^{2}d\xi\Big\}^{\theta}\\
&=\|u\|^{2(1-\theta)}_{H^{r_0}(\mathbb{R})} \, \|u\|^{2\theta}_{H^{r_1}(\mathbb{R})}.
\end{split}\end{equation*}
This completes the proof.
\end{proof}

\subsection{Error estimate of orthogonal projections}


Define the approximation space
\begin{equation}\label{VNlambda}
\begin{split}
V_{N}^{\lambda} & =\big\{\phi(x):\phi(x)=S(t)P(t),\;\; \forall  P\in {\mathcal P}_N \big\}\\&=\text{span}\big\{R_{n}^{\lambda}\left(x\right): n=0,1,...,N\big\}.
\end{split}
\end{equation}
Consider the $L^2$-orthogonal projection $\pi_{N}^{\lambda}: L^2(\mathbb R)\to V_N^\lambda,$ i.e.,
\begin{equation}\label{phicase}
\pi_{N}^{\lambda} u(x)=\sum_{n=0}^N \hat u_n^\lambda R_n^\lambda(x),\quad\;\; \hat u_n^\lambda=\frac 1 {\gamma_n^\lambda}
\int_{\mathbb R} u(x) R_n^\lambda(x)dx.
\end{equation}

For notational convenience, we introduce the pairs of functions associated with the mapping \eqref{algebraicmapping}:
\begin{equation}\label{pairfun}
\begin{split}
  &u(x)=U(t(x)),\quad \breve u(x)=\frac{u(x)}{s(x)}=\frac{U(t)}{S(t)}=\breve U(t),\;\;\; {\rm where} \\
  & s(x):=\frac{1} {(1+x^2)^{(\lambda+1)/2}}= (1-t^2)^{(\lambda+1)/2}:=S(t).
  \end{split}
\end{equation}
In what follows, the notation with or without ``\, $\breve{}$\, '' has the same meaning.

In order to describe the approximation errors, we introduce  new differential operators as follows
\begin{equation}\label{diffopts}
\begin{split}
  & {\mathscr D}_x u:=a(x)\frac{d\breve u}{dx}=\frac{d\breve U}{dt},\quad
  {\mathscr D}_x^2 u:=a(x) \frac{d}{dx}\Big\{a(x)\frac{d\breve u}{dx}\Big\}=\frac{d^2\breve U}{dt^2},\cdots, \\
   & {\mathscr D}_x^ku=a(x)\frac{d}{dx}\Big\{a(x)\frac {d}{dx}\Big\{\cdots\Big\{a(x)\frac{d\breve u}{dx}\Big\}\cdots\Big\}\Big\}=\frac{d^k\breve U}{dt^k},
\end{split}
\end{equation}
where  $ a(x)={dx}/{dt}=(1+x^2)^{\frac 3 2}.$
Correspondingly, we define the Sobolev space
\begin{equation}\label{Sobolvm}
\begin{split}
  {\mathbb B}^m_{\lambda}(\mathbb R)=\big\{u: u \,\, \text{is measurable in $\mathbb R$ and}\, \, \|u\|_{{\mathbb B}^m_{\lambda}(\mathbb R)}<\infty\big\},
  \end{split}
\end{equation}
 equipped with the norm and semi-norm
\begin{equation}\label{seminorm}
\begin{split}
 & \|u\|_{{\mathbb B}^m_{\lambda}(\mathbb R)}=\Big(\sum_{k=0}^m\big\|(1+x^2)^{\frac 1 4-\frac{\lambda+m} 2}{\mathscr D}_x^ku \big\|_{L^2(\mathbb R)}^2\Big)^{\frac 1 2},\quad \\
 & |u|_{{\mathbb B}^m_{\lambda}(\mathbb R)}=\big\|(1+x^2)^{\frac 1 4-\frac{\lambda+m} 2}{\mathscr D}_x^m u \big\|_{L^2(\mathbb R)}
    \end{split}
\end{equation}

\begin{theorem}\label{maintheoremest} For any $u\in H^s(\mathbb R)\cap {\mathbb B}^m_{\lambda}(\mathbb R)$ with integer $ 1\le m\le N+1, s\in (0,1),$ and $\lambda>-1/2,$ we have
  \begin{equation}\label{orthNs2}
 \| \pi_{N}^{\lambda} u-u \|_{H^s(\mathbb R)}\le cN^{s-m} |u|_{{\mathbb B}^m_{\lambda}(\mathbb R)},
\end{equation}
where $c$ is a positive constant independent of $N$ and $u.$
\end{theorem}
\begin{proof}   We take two steps to carry out the proof.

{\sc Step 1}:  We first prove that
\begin{equation}\label{mianResult}
\big\|\sqrt{1+x^2}\, (\pi_{N}^{\lambda} u-u)'\big\|_{L^2(\mathbb R)}+N\|\pi_{N}^{\lambda} u-u\|_{L^2(\mathbb R)}
\le cN^{1-m} |u|_{{\mathbb B}^m_{\lambda}(\mathbb R)},
\end{equation}
For this purpose, we  study the close relation between  $ \pi_{N}^{\lambda} $ and the orthogonal projection $\varPi_N^\lambda:  L^2_{\omega_\lambda}(I)\to {\mathcal P}_N,$ such that   for any $\varPhi\in  L^2_{\omega_\lambda}(I), $
\begin{equation}\label{Gen-proj}
\int_{-1}^1 ( \varPi_N^\lambda \varPhi(t)-\varPhi(t))\varPsi(t) \omega_\lambda(t)\, dt=0,\quad \forall\,\varPsi\in {\mathcal P}_N.
\end{equation}
Recall the Gegenbauer polynomial approximation result (cf. \cite[Thm 3.55]{SWT}):  {\em if $\varPhi^{(l)}(t)\in L^2_{\omega_{\lambda+l}}(I)$ for $0\le l\le m,$ we have
\begin{equation}\label{genpolyapp}
  \|(\varPi_N^\lambda \varPhi-\varPhi)^{(l)} \|_{L^2_{\omega_{\lambda+l}}(I)}\le c N^{l-m}\|\varPhi^{(m)}\|_{L^2_{\omega_{\lambda+m}}(I)},
\end{equation}
where the weight function  $\omega_a(t)=(1-t^2)^{a-1/2}.$}

From    \eqref{phicase} and \eqref{pairfun}, we find
\begin{equation}\label{hatunUn}
\begin{split}
\hat u_n & =\frac 1 {\gamma_n^\lambda}\int_{\mathbb R} u(x) R_n^\lambda(x)dx=
\frac 1 {\gamma_n^\lambda} \int_{-1}^1 U(t)S(t) C_n^\lambda(t) \frac{dx}{dt} dt
\\& =\frac 1 {\gamma_n^\lambda} \int_{-1}^1 \frac{U(t)}{S(t)} C_n^\lambda(t) (1-t^2)^{\lambda-1/2} dt
= \frac 1 {\gamma_n^\lambda} \int_{-1}^1 \breve{U}(t) C_n^\lambda(t)\omega_\lambda(t)dt=
\widehat {\breve U}_n.
\end{split}
\end{equation}
Therefore, we have
\begin{equation}\label{relatDP2}
\begin{split}
 e_N(x)& :=u(x)-\pi_N^{\lambda} u(x)=\sum_{n=N+1}^\infty \hat u_n R_n^\lambda(x)=S(t)\sum_{n=N+1}^\infty
\widehat {\breve U}_n C_n^\lambda(t)\\
&  =S(t)\big(\breve U(t)-\varPi_N^\lambda \breve U(t)\big):=S(t) \breve e_N(t).
\end{split}
\end{equation}
As a result, there holds
\begin{equation}\label{Bfun33}
  \int_{\mathbb R} |e_N(x)|^2dx=\int_{-1}^1 |S(t) \breve e_N(t)|^2 \frac{dx}{dt} dt
   =\int_{-1}^1 |\breve e_N(t)|^2 \omega_\lambda(t) dt.
\end{equation}
Thus, using \eqref{genpolyapp}  with $l=0$, we derive from  \eqref{diffopts}-\eqref{seminorm} that
\begin{equation}\label{relatDP33}
 \|e_N\|_{L^2(\mathbb R)}=\|\breve e_N\|_{L^2_{\omega_\lambda}(I)}\le
c N^{-m}\|\partial_t^m \breve U\|_{L^2_{\omega_{\lambda+m}}(I)}=cN^{-m}|u|_{{\mathbb B}^m_{\lambda}(\mathbb R)}.
\end{equation}

Like \eqref{relatDP2},  we can show
\begin{equation}\label{relatDP200}
\begin{split}
 e_N'(x)& =\sum_{n=N+1}^\infty \hat u_n \frac{dR_n^\lambda}{dx} (x)=\sum_{n=N+1}^\infty
\widehat {\breve U}_n \frac{d}{dt} \big(S(t) C_n^\lambda(t)\big) \frac{dt}{dx}\\
&  =\big(S(t) \breve e_N'(t)
+S'(t) \breve e_N(t)\big) \frac{dt}{dx}\\
&= (1-t^2)^{\frac \lambda 2+2}  \breve e_N'(t)-(\lambda+1) t (1-t^2)^{\frac{\lambda} 2+1}   \breve e_N(t).
\end{split}
\end{equation}
Similar to \eqref{Bfun33}, we derive from   \eqref{genpolyapp}  with $l=0,1$ that
\begin{equation*}\label{Bfun44}
\begin{split}
  \int_{\mathbb R} |e_N'(x)|^2 (1+x^2) dx & \le 2
  \int_{-1}^1 |\breve e_N'(t)|^2 \omega_{\lambda+2}(t) dt
  + 2(1+\lambda)^2 \int_{-1}^1 |\breve e_N(t)|^2 \omega_{\lambda}(t) dt\\
  &\le c N^{2-2m}  \|\partial_t^m \breve U\|_{L^2_{\omega_{\lambda+m}}(I)}^2=c N^{2-2m} |u|_{{\mathbb B}^m_{\lambda}(\mathbb R)}^2.
   \end{split}
\end{equation*}
Then the estimate \eqref{mianResult} is a direct consequence of \eqref{relatDP33} and the above.

\smallskip

 {\sc Step 2}:  It is evident that the result \eqref{mianResult} implies
\begin{equation}\label{mianResult2}
\begin{split}
\|\pi_{N}^{\lambda} u-u\|_{H^1(\mathbb R)} & \le  \big\|\sqrt{1+x^2}\, (\pi_{N}^{\lambda} u-u)'\big\|_{L^2(\mathbb R)}+\|\pi_{N}^{\lambda} u-u\|_{L^2(\mathbb R)}\\
& \le cN^{1-m} |u|_{{\mathbb B}^m_{\lambda}(\mathbb R)},
\end{split}
\end{equation}
and
\begin{equation}\label{L2est}
\|\pi_{N}^{\lambda} u-u\|_{L^2(\mathbb R)}
\le cN^{-m} |u|_{{\mathbb B}^m_{\lambda}(\mathbb R)}.
\end{equation}
Then using the interpolation inequality in Lemma \ref{interpola} with $r_0=0,r_1=1$ and $\theta=s,$ we obtain from \eqref{mianResult2}-\eqref{L2est} that
\begin{equation}\label{orthNs22}
 \| \pi_{N}^{\lambda} u-u \|_{H^s(\mathbb R)}\le\|\pi_{N}^\lambda u-u \|_{L^2(\mathbb R)}^{1-s}\,
   \|\pi_{N}^\lambda u-u \|_{H^1(\mathbb R)}^{s}\le cN^{s-m} |u|_{{\mathbb B}^m_{\lambda}(\mathbb R)}.
\end{equation}
This completes the proof.
\end{proof}

 In the error analysis, it is necessary to consider the $H^s$-orthogonal projection.
 Define the bilinear form on $H^s({\mathbb R}):$
\begin{equation}\label{fbilinearA}
a_s(u,v)=\big((-\Delta)^{s/2} u, (-\Delta)^{s/2} v\big)+(u, v).
\end{equation}
Consider the orthogonal projection  $\pi_{N,\lambda}^s: H^s(\mathbb R)\to V_N^\lambda$ such that
\begin{equation}\label{bilinear}
a_s(\pi_{N,\lambda}^s u-u, v\big)=0,\quad \forall\, v\in  V_N^\lambda.
\end{equation}
Then by the projection theorem, we have
\begin{equation}\label{orthproj}
\| \pi_{N,\lambda}^{s} u-u \|_{H^s(\mathbb R)}=\inf_{\phi\in V_N^\lambda} \|\phi-u \|_{H^s(\mathbb R)}.
\end{equation}
Taking $\phi=\pi_N^\lambda u,$ we immediately derive the following estimate.
\begin{theorem}\label{maintheoremest2} For any $u\in H^s(\mathbb R)\cap {\mathbb B}^m_{\lambda}(\mathbb R)$ with integer $ 1\le m\le N+1, s\in (0,1),$ and $\lambda>-1/2,$ we have
  \begin{equation}\label{orthNs23}
 \| \pi_{N,\lambda}^{s} u-u \|_{H^s(\mathbb R)}\le cN^{s-m} |u|_{{\mathbb B}^m_{\lambda}(\mathbb R)},
\end{equation}
where $c$ is a positive constant independent of $N$ and $u.$
\end{theorem}

%
%

%
%
%
%
%
%
%

\subsection{Error estimate of  interpolation} Let  $\{t_j^\lambda,\rho_j^\lambda\}_{j=0}^N$ be the Gegenbauer-Gauss quadrature
nodes and weights,   where  $\{t_j^\lambda\}$  are zeros of the Gegenbauer polynomial $C_{N+1}^{\lambda}(t).$
Define the mapped nodes and weights:
\begin{equation}\label{mappedxj}
 x_j^\lambda=\frac{t_j^\lambda}{\sqrt{1-(t_j^\lambda)^2}},\quad \omega_j^\lambda= (1+(t_j^\lambda)^2)^{-\lambda} \rho_j^\lambda,\quad 0\le j\le N.
\end{equation}
Then by the exactness of the Gagenbauer-Gauss quadrature (cf. \cite[Ch 3]{SWT}), we have
\begin{equation*}\label{guassmappedjacobi}
\begin{split}
\int_{\mathbb R}u(x) v(x)\, dx & =\int_{-1}^1 \frac{U(t) V(t)}{(1-t^2)^{3/2}}  dt=\int_{-1}^1 \frac{U(t)}{S(t)} \frac{U(t)}{S(t)}  (1-t^2)^{\lambda-1/2} dt   \\&
=\sum_{j=0}^{N}\frac{U(t_{j}^\lambda)} {S(t_j^\lambda)}  \frac{V(t_{j}^\lambda)} {S(t_j^\lambda)}  \rho_{j}^\lambda,\qquad {\rm if}\quad  \frac{U(t)}{S(t)} \cdot  \frac{V(t)}{S(t)}  \in {\mathcal  P}_{2N+1},
\end{split}
\end{equation*}
which, together with \eqref{VNlambda},  implies the exactness of quadrature
\begin{equation}\label{gujacobi}
\begin{split}
\int_{\mathbb R}u(x) v(x)\, dx =\sum_{j=0}^N u(x_j^\lambda)v(x_j^\lambda) \omega_j^\lambda,\quad \forall\, u\cdot v\in V_{2N+1}^\lambda.
\end{split}
\end{equation}

We now introduce  the  interpolation operator  $I_{N}^{\lambda}u\,:\, C(\mathbb R)\to  V_{N}^{\lambda} $ such that
\begin{equation*}
 I_{N}^{\lambda}u(x_{j}^\lambda)=u(x_{j}^\lambda), \quad 0\le j\le N.
\end{equation*}
As a consequence of  \eqref{orthss} and \eqref{gujacobi}, we have
\begin{equation}\label{INlambda}
I_{N}^{\lambda}u(x)=\sum_{n=0}^{N}\tilde{u}_{n}^{\lambda}R_{n}^{\lambda}(x), \quad {\rm where}\quad
\tilde{u}_{n}^{\lambda}=\frac{1}{\gamma_{n}^{\lambda}}\sum_{j=0}^{N}u(x_{j}^\lambda)R_{n}^{\lambda}(x_j^\lambda)\omega_{j}^\lambda.
\end{equation}

We have the following interpolation approximation result.
\begin{theorem}\label{intpol-thm} For any $u\in H^s(\mathbb R)\cap {\mathbb B}^m_{\lambda}(\mathbb R)$ with integer $ 1\le m\le N+1, s\in (0,1),$ and $\lambda>-1/2,$ we have
  \begin{equation}\label{Intesta}
 \| I_{N}^{\lambda} u-u \|_{H^s(\mathbb R)}\le cN^{s-m} |u|_{{\mathbb B}^m_{\lambda}(\mathbb R)},
\end{equation}
where $c$ is a positive constant independent of $N$ and $u.$
\end{theorem}
\begin{proof}
Recall the Gegenbauer-Gauss interpolation 
  $I_N^G: C(-1,1)\to {\mathcal P}_N,$ such that
\begin{equation*}
 I_{N}^{G} U(t_{j}^\lambda)=U(t_{j}^\lambda),  \quad 0\le j\le N.
\end{equation*}
Then we have the expansion
\begin{equation}\label{IGlambda}
I_{N}^{G}U(t)=\sum_{n=0}^{N}\tilde{U}_{n}^{\lambda}C_{n}^{\lambda}(t), \quad {\rm where}\quad
\tilde{U}_{n}^{\lambda}=\frac{1}{\gamma_{n}^{\lambda}}\sum_{j=0}^{N}U(t_{j}^\lambda)C_{n}^{\lambda}(t_j^\lambda)\rho_{j}^\lambda.
\end{equation}
One verifies from \eqref{modifiedrational}, \eqref{pairfun},  \eqref{mappedxj} and \eqref{INlambda}-\eqref{IGlambda} that
\begin{equation}\label{interela}
I_{N}^{\lambda}u(x)=S(t)\, I_{N}^{G}\Big\{\frac{U(t)}{S(t)}\Big\}=S(t)\, I_{N}^{G} \breve U(t).
\end{equation}
Thus,
\begin{equation}\label{relatDP2s}
\begin{split}
 e_N(x)& :=u(x)-I_N^{\lambda} u(x)=  S(t)\big(\breve U(t)-I_N^G \breve U(t)\big):=S(t) \breve e_N(t),
\end{split}
\end{equation}
where with a little abuse of notation, we still use  the same notation as in \eqref{relatDP2}.
Following the lines as in \eqref{Bfun33}-\eqref{Bfun44}, we can show that
\begin{equation}\label{estA1}
\begin{split}
 & \| I_{N}^{\lambda} u-u\|_{L^2(\mathbb R)}=\|I_N^G \breve U-\breve U\|_{L^2_{\omega_\lambda}(I)},
\end{split}
\end{equation}
and
\begin{equation}\label{estA2}
\begin{split}
 &  \| (I_{N}^{\lambda} u-u)'\|_{L^2(\mathbb R)}\le c\big(\|I_N^G \breve U-\breve U\|_{L^2_{\omega_\lambda}(I)} + \| \sqrt{1-t^2} (I_N^G \breve U-\breve U)'\|_{L^2_{\omega_\lambda}(I)} \big).
\end{split}
\end{equation}
According to \cite[Thm. 3.41]{SWT} on the Gegenbauer-Gauss interpolation error estimate, we have
\begin{equation*}\label{estA3}
\begin{split}
N \|I_N^G \breve U-\breve U\|_{L^2_{\omega_\lambda}(I)}  & + \| \sqrt{1-t^2} (I_N^G \breve U-\breve U)'\|_{L^2_{\omega_\lambda}(I)}
\\& \le  c N^{1-m}  \|\partial_t^m \breve U\|_{L^2_{\omega_{\lambda+m}}(I)}.
\end{split}
\end{equation*}
Then by the interpolation inequality in Lemma \ref{interpola}, we obtain from the above that
\begin{equation*}\label{orthNs20}
 \| I_{N}^{\lambda} u-u \|_{H^s(\mathbb R)}\le\|I_{N}^{\lambda} u-u \|_{L^2(\mathbb R)}^{1-s}\,
   \|I_{N}^{\lambda} u-u \|_{H^1(\mathbb R)}^{s}\le cN^{s-m} |u|_{{\mathbb B}^m_{\lambda}(\mathbb R)}.
\end{equation*}
This completes the proof.
\end{proof}

\section{Modified rational spectral-Galerkin methods}\label{sect:5}

In this section, we consider the spectral-Galerkin approximation to a model equation, and conduct the error analysis.
We also present some numerical results to show our proposed method outperforms the Hermite approximations in
\cite{MaoShen,hermitecollocation}.

\subsection{The scheme and its convergence} Consider the model equation
\begin{equation}\label{modeqn}
\begin{cases}
(-\Delta)^{\alpha/2}u(x) +\rho u(x) = f(x), \quad &x\in\mathbb{R},\\
  u(x)=0,\quad &|x| \to \infty,
  \end{cases}
\end{equation}
for $\alpha\in (0,2),$ where $f\in L^2(\mathbb R)$ and the constant $\rho>0.$

For notational convenience, let $s=\alpha/2.$  A weak form of \eqref{modeqn} is to find $u\in H^s(\mathbb R)$ such that
\begin{equation}\label{wkformA}
 \tilde a_s(u,v):=\big((-\Delta)^{s/2} u, (-\Delta)^{s/2} v\big)+\rho (u, v)
 =(f,v),\quad \forall v\in H^s(\mathbb R).
\end{equation}
The spectral-Galerkin scheme is to find $u_N\in V_N^\lambda$ (defined in \eqref{VNlambda}) such that
\begin{equation}\label{schemeA}
 \tilde a_s(u_N,v_N)
 =(I_N^\lambda f,v_N),\quad \forall v_N\in V_N^\lambda.
\end{equation}

Denote $e_N=u_N-\pi_{N,\lambda}^s u$ and $\tilde e_N=u-\pi_{N,\lambda}^s u.$  By a standard analysis, we find that for any $v_N\in V_N^\lambda,$
\begin{equation*}
 \begin{split}
 \tilde a_s(e_N,v_N)
 &=\tilde a_s(\tilde e_N,v_N) +  (I_N^\lambda f-f,v_N) \\
 &= a_s(\tilde e_N,v_N) +(\rho-1)(\tilde e_N,v_N)  +  (I_N^\lambda f-f,v_N) \\
 &= (\rho-1)(\tilde e_N,v_N)  +  (I_N^\lambda f-f,v_N).
 \end{split}
\end{equation*}
Taking $v_N=e_N$ and using the Cauchy-Schwarz inequality, we obtain
\begin{equation*}
\|e_N\|_{H^s(\mathbb R)}^2\le c \big(\|\tilde e_N\|_{L^2(\mathbb R)}^2+ \|I_N^\lambda f-f\|_{L^2(\mathbb R)}^2\big).
\end{equation*}
Thus, by the triangle inequality, we derive
\begin{equation}\label{vNeN}
\begin{split}
\|u-u_N\|_{H^s(\mathbb R)} & \le c \big(\|\tilde e_N\|_{H^s(\mathbb R)}+ \|I_N^\lambda f-f\|_{L^2(\mathbb R)}\big)\\
&\le  cN^{s-m} |u|_{{\mathbb B}^m_{\lambda}(\mathbb R)}+ cN^{-k} |f|_{{\mathbb B}^{k}_{\lambda}(\mathbb R)}.
\end{split}
\end{equation}
In summary, we have the following convergence result.
\begin{theorem}\label{Conv-thm} For any $u\in H^s(\mathbb R)\cap {\mathbb B}^m_{\lambda}(\mathbb R)$ and $f\in {\mathbb B}^k_{\lambda}(\mathbb R) $ with integer $ 1\le m, k\le N+1, s=\alpha/2\in (0,1),$ and $\lambda>-1/2,$ we have
  \begin{equation}\label{IntestaS}
 \| u-u_N\|_{H^s(\mathbb R)}\le cN^{s-m} |u|_{{\mathbb B}^m_{\lambda}(\mathbb R)}+ cN^{-k} |f|_{{\mathbb B}^{k}_{\lambda}(\mathbb R)},
\end{equation}
where $c$ is a positive constant independent of $N$ and $u,f.$
\end{theorem}

\subsection{Numerical examples}
We now present several examples to show the convergence behaviour of the above spectral Galerkin method.
In all tests, we report the numerical errors in the $L^2$-norm, and set $\rho=1$. Here,  we only
consider the cases with $\lambda=0$ and $\lambda=0.5$, which correspond to the modified mapped Chebyshev rational functions and modified mapped Legendre functions, respectively.

\smallskip

{\bf Example 1:}  {\em Exponential decay  $f(x).$} \,We first consider equation \eqref{modeqn}  with  $f(x)=\exp(-{x^2}/{2})(1+x).$ Since the closed-form exact solution is not available, we take the numerical solution with $N=600$ as the reference solution. The convergence results with MMGFs for $\alpha=0.4, 1, 1.6$ are presented in Figure \ref{galerkinexponential1} (middle and right). In the left plot, we have also presented the convergence results for the Hermite function approach in \cite{MaoShen}.  It is clearly seen that the MMGFs approach outperforms the Hermite approximations for all cases, namely, the MMGFs approach admits much higher convergence rates. This can also be seen form Table 1, where we have presented the order of convergence for both approaches.

\begin{figure}[!t]
\centering
\begin{minipage}[c]{0.33\textwidth}
\centering
\includegraphics[height=4cm,width=4.3cm]{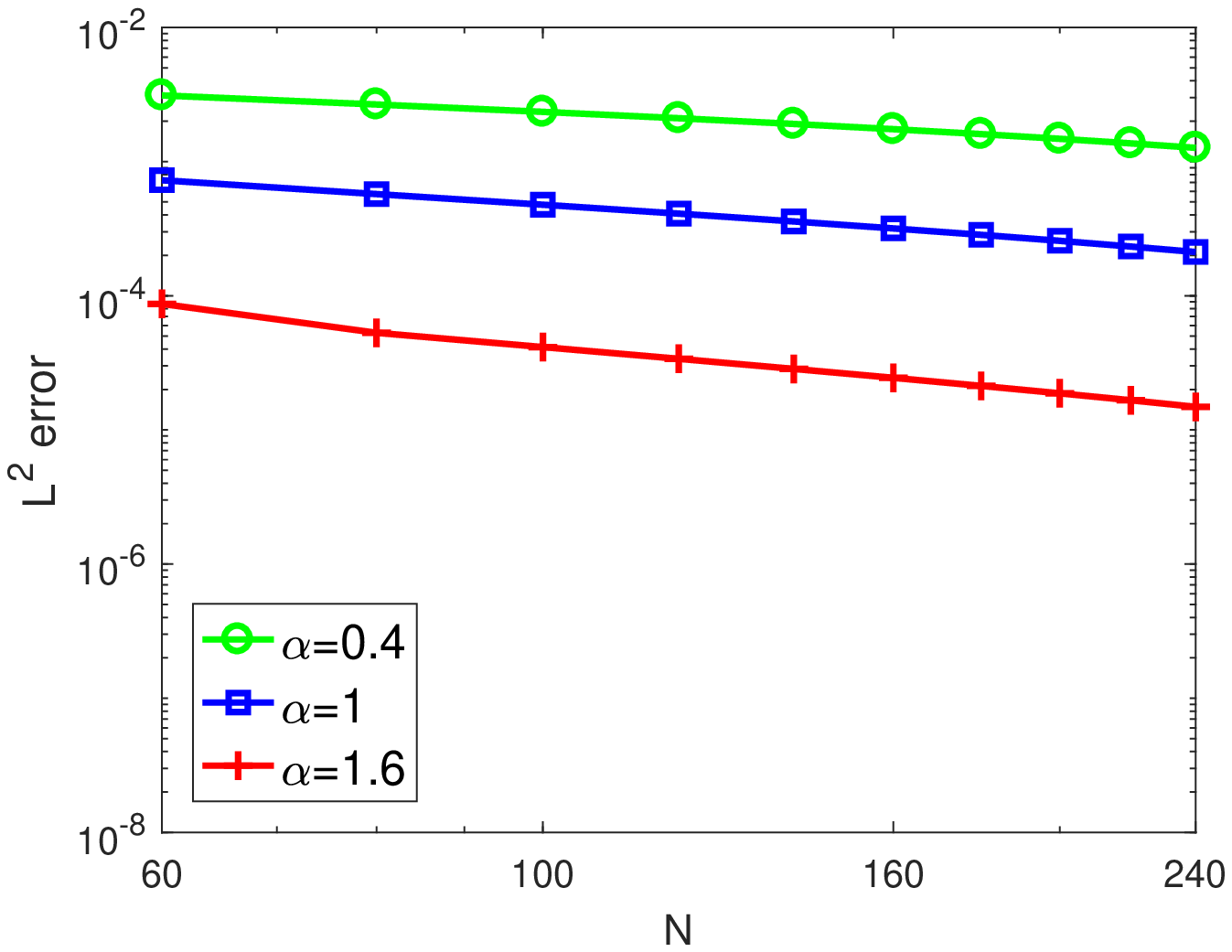}
\end{minipage}%
\begin{minipage}[c]{0.33\textwidth}
\centering
\includegraphics[height=4cm,width=4.3cm]{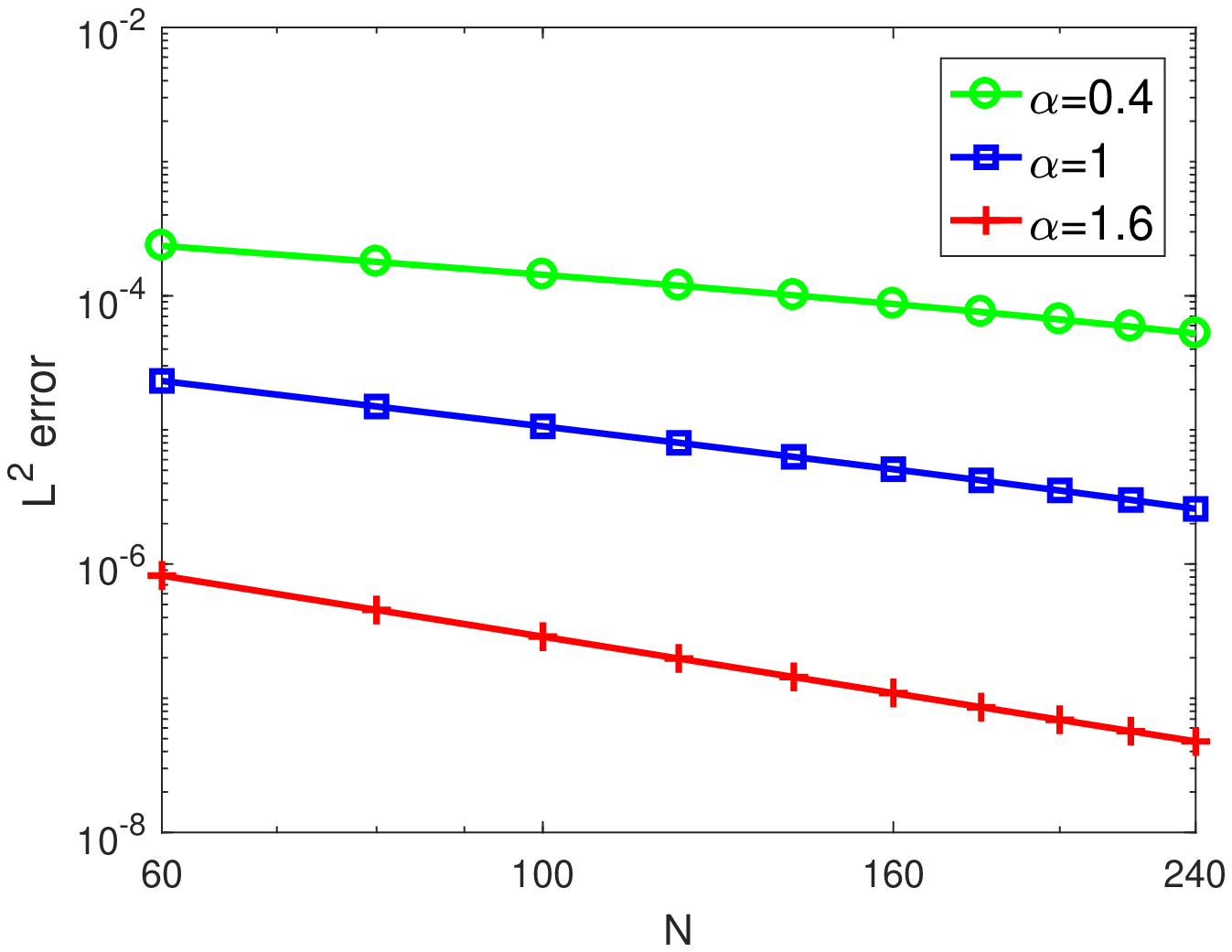}
\end{minipage}
\begin{minipage}[c]{0.33\textwidth}
\centering
\includegraphics[height=4cm,width=4.3cm]{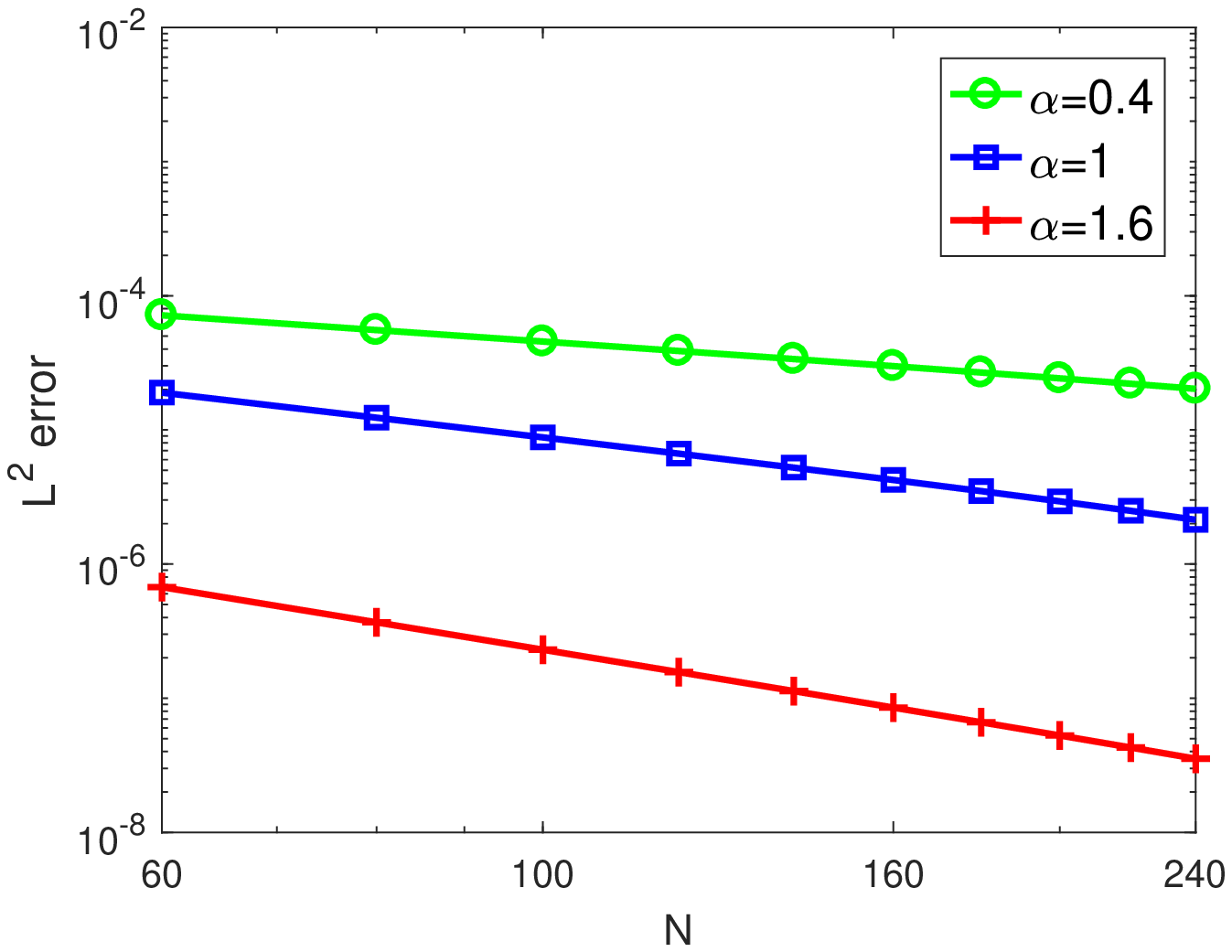}
\end{minipage}
\caption{$L^2$-error with $f(x)=\exp(-{x^2}/{2})(1+x)$. Left: Hermite function approach in \cite{MaoShen} with scaling factor $1/0.4$. Middle: MMGFs approach with $\lambda=0$ and scaling factor $\mu=5$. Right: MMGFs approach with $\lambda=0.5$ and scaling factor $\mu=5$.}
\label{galerkinexponential1}
\end{figure}

\begin{table}[!t]\label{tableA}
\caption{Rate of convergence using the generalized Hermite function, MMGFs with $\lambda=0$,  and MMGFs with $\lambda=0.5,$ $\alpha=1$ and $f(x)=\exp(-{x^2}/{2})(1+x).$ }
\begin{center}
\begin{tabular}{|c|c|c|c|c|c|c|}
\hline
 &\multicolumn{2}{|c|}{Hermite} & \multicolumn{2}{|c|}{MMGF $\lambda=0$}&\multicolumn{2}{|c|}{MMGF $\lambda=0.5$}\\
\hline
$N$ &$L^2$ error&Order&$L^2$ error&Order&$L^2$ error&Order\\
\hline
60 &7.26e-4&  --  & 2.32e-5& --&1.90e-5    & --\\
\hline
80 &5.72e-4& 0.83 &1.50e-5& -1.52&1.23e-5 &1.50     \\
\hline
100  &4.77e-4& 0.82&1.06e-5&-1.53&8.79e-6 &1.51    \\
\hline
120  &4.09e-4& 0.84 &8.02e-6&-1.55&6.65e-6&1.53     \\
\hline
140  &3.58e-4&  0.87&6.30e-6&-1.57&5.23e-6&1.56     \\
\hline
160  &3.17e-4&  0.90&5.09e-6&-1.59&4.23e-6&1.59       \\
\hline
180 &2.84e-4& 0.94 & 4.21e-6&-1.62&3.50e-6&1.62       \\
\hline
200 &2.57e-4& 0.98 &3.53e-6& -1.66&2.94e-6&1.66       \\
\hline
220  &2.32e-4 &1.03&3.01e-6& -1.70&2.49e-6&1.71        \\
\hline
240  &2.12e-4 & 1.08&2.58e-6&-1.75&2.14e-6&1.77        \\
\hline

 \end{tabular}
\end{center}
\label{convergencerate}
\end{table}


\smallskip

{\bf Example 2:}  {\em Algebraic decay   $f(x).$}\,  We next consider equation \eqref{modeqn}  with an algebraic decay source term: $f\left(x\right)=\frac{1}{(1+x^2)^2}.$ The plots of the error decay for both  Hermite functions and MMGFs are  in Figure \ref{galerkinalgebraic1}. Indeed, we observe the convergence behaviour similar to the previous example -- the MMGF approach has a much better performance.
The comparison in Table 2 also shows that the proposed approach converges much faster than the Hermite method.
\begin{figure}[!t]
\centering
\begin{minipage}[c]{0.33\textwidth}
\centering
\includegraphics[height=4cm,width=4.3cm]{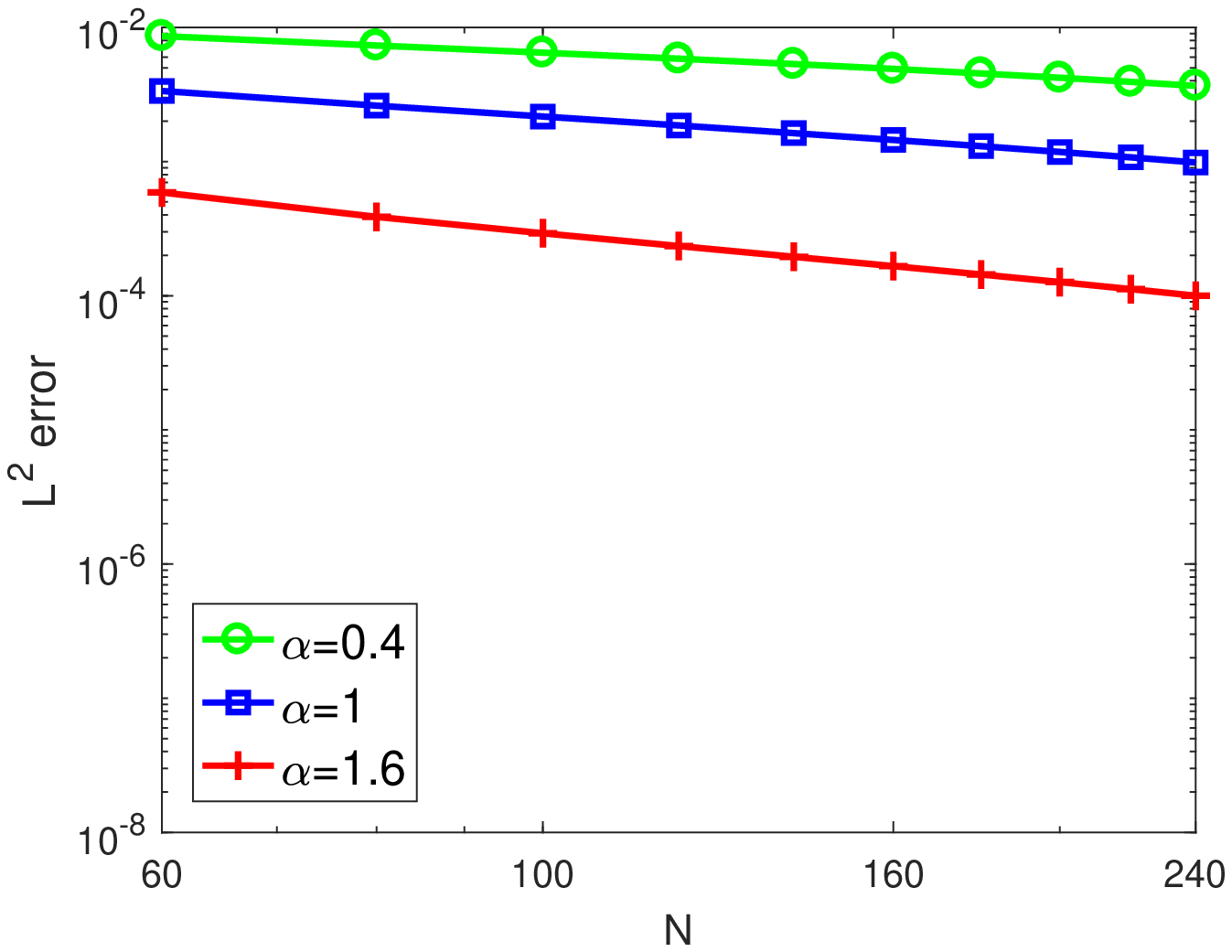}
\end{minipage}%
\begin{minipage}[c]{0.33\textwidth}
\centering
\includegraphics[height=4cm,width=4.3cm]{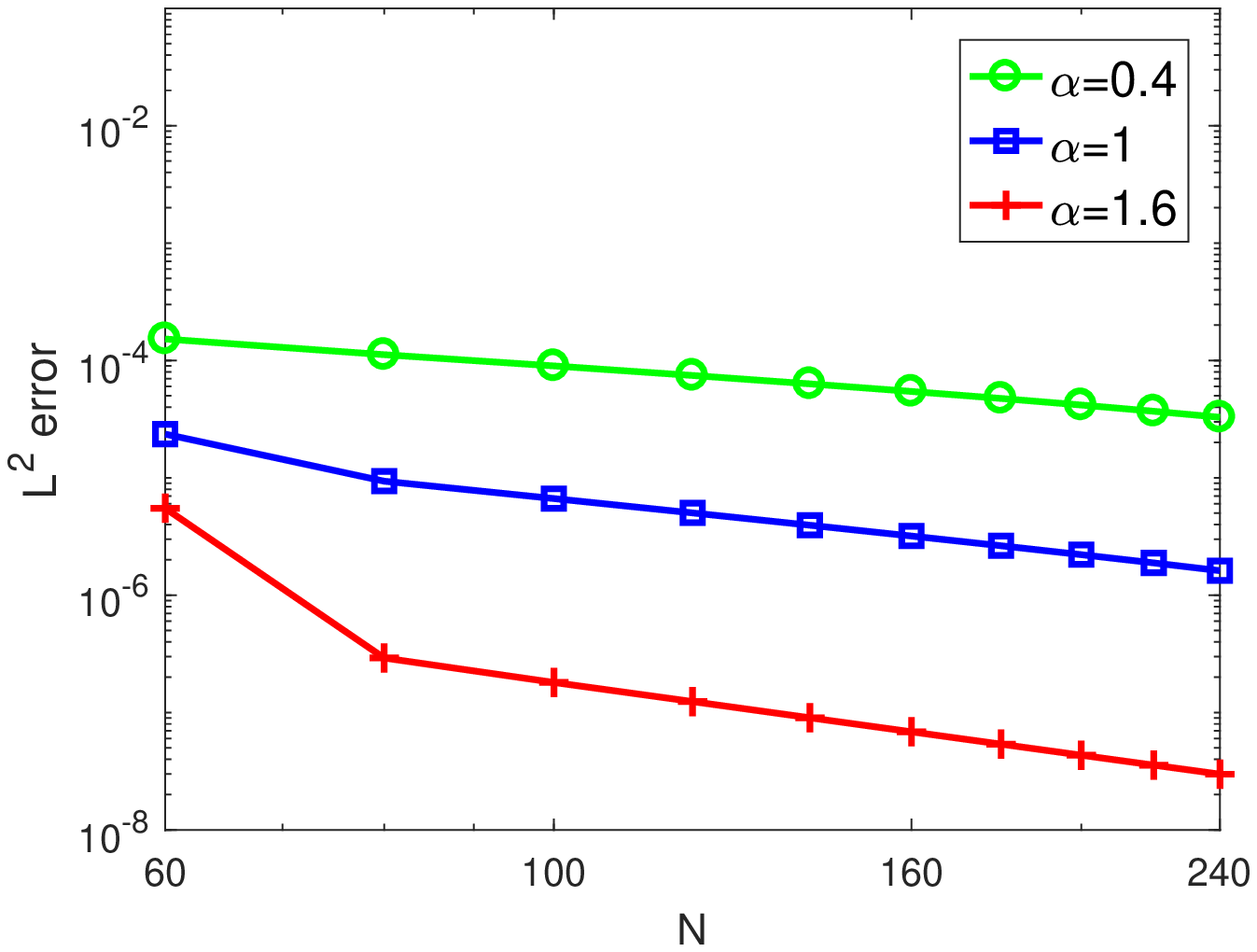}
\end{minipage}
\begin{minipage}[c]{0.33\textwidth}
\centering
\includegraphics[height=4cm,width=4.3cm]{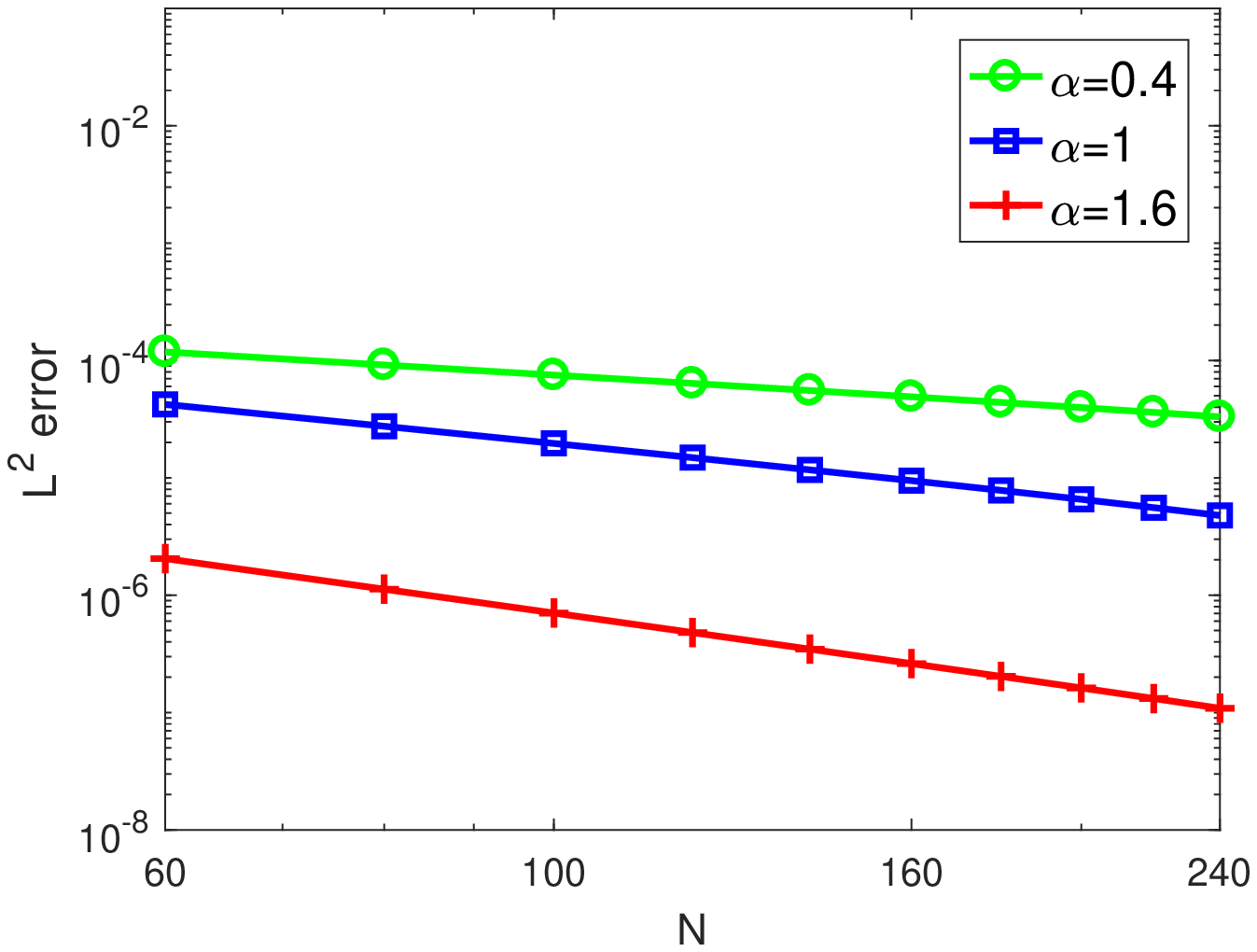}
\end{minipage}
\caption{$L^2$-error with $f(x)=\frac{1}{(1+x^2)^2}$. Left: Hermite function approach in \cite{MaoShen} with scaling factor $1/0.7$. Middle: The MMGFs approach with $\lambda=0$ and scaling factor $\mu=3$. Right: The MMGFs approach with $\lambda=0.5$ and scaling factor $\mu=3$.}
\label{galerkinalgebraic1}
\end{figure}

\begin{table}[!t]\label{tableB}
\caption{Rate of convergence using the generalized Hermite function, MMGFs with $\lambda=0$,  and MMGFs with $\lambda=0.5.$  $\alpha=1$ and $f\left(x\right)=\frac{1}{(1+x^2)^2}$.}
\begin{center}
\begin{tabular}{|c|c|c|c|c|c|c|}
\hline
&\multicolumn{2}{|c|}{Hermite} & \multicolumn{2}{|c|}{MMGF $\lambda=0$}&\multicolumn{2}{|c|}{MMGF $\lambda=0.5$}\\
\hline
$N$ &$L^2$ error&Order&$L^2$ error&Order&$L^2$ error&Order\\
\hline
60 &3.36e-3&  --  & 2.36e-5& --&4.23e-5    & --\\
\hline
80 &2.61e-3& 0.87 &9.36e-6& -3.21&2.75e-5 &1.49     \\
\hline
100  &2.17e-3& 0.84&6.66e-5&-1.53&1.96e-5 &1.51    \\
\hline
120  &1.86e-3& 0.84 &5.02e-6&-1.55&1.49e-5&1.53     \\
\hline
140  &1.63e-3& 0.86&3.95e-6&-1.56&1.17e-5&1.56     \\
\hline
160  &1.45e-3& 0.88&3.19e-6&-1.59&9.46e-6&1.59       \\
\hline
180 &1.30e-3& 0.91 &2.64e-6&-1.62&7.81e-6&1.62       \\
\hline
200 &1.18e-3& 0.94 &2.21e-6& -1.66&6.56e-6&1.66       \\
\hline
220  &1.08e-3 &0.97&1.88e-6& -1.70&5.57e-6&1.71        \\
\hline
240  &9.84e-4 & 1.01&1.62e-6&-1.75&4.78e-6&1.77        \\
\hline

 \end{tabular}
\end{center}
\label{convergencerate1}
\end{table}

\smallskip

To better understand the solution behaviours, we  present in Figure \ref{uchebyshev} the asymptotic behavior of the  ``exact'' solutions as $|x|\gg 1$ for the above two examples.  We see that, for both examples with very different decay of $f(x)$, the solution $u(x)$ decays at the same rate: $\lvert x\rvert^{-\alpha-1}$. This testifies the solution decays at a rate of  a power law, as opposite to the usual Laplacian.
 This also  explains the reason why MMGFs have a better performance than the Hermite functions.

 \begin{figure}[!t]
\centering
\begin{minipage}[c]{0.45\textwidth}
\centering
\includegraphics[height=4.2cm,width=5.5cm]{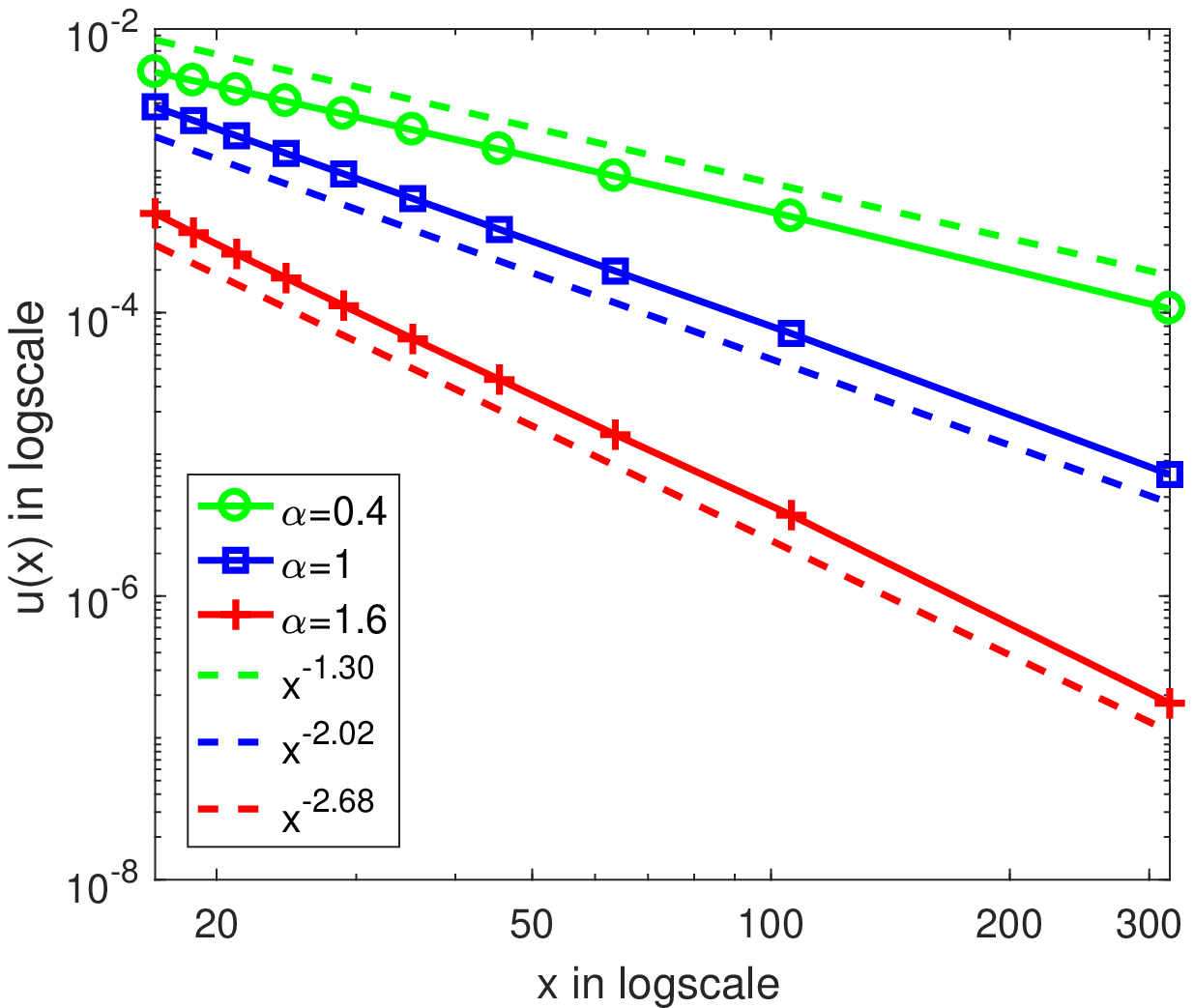}
\end{minipage}%
\begin{minipage}[c]{0.45\textwidth}
\centering
\includegraphics[height=4.2cm,width=5.5cm]{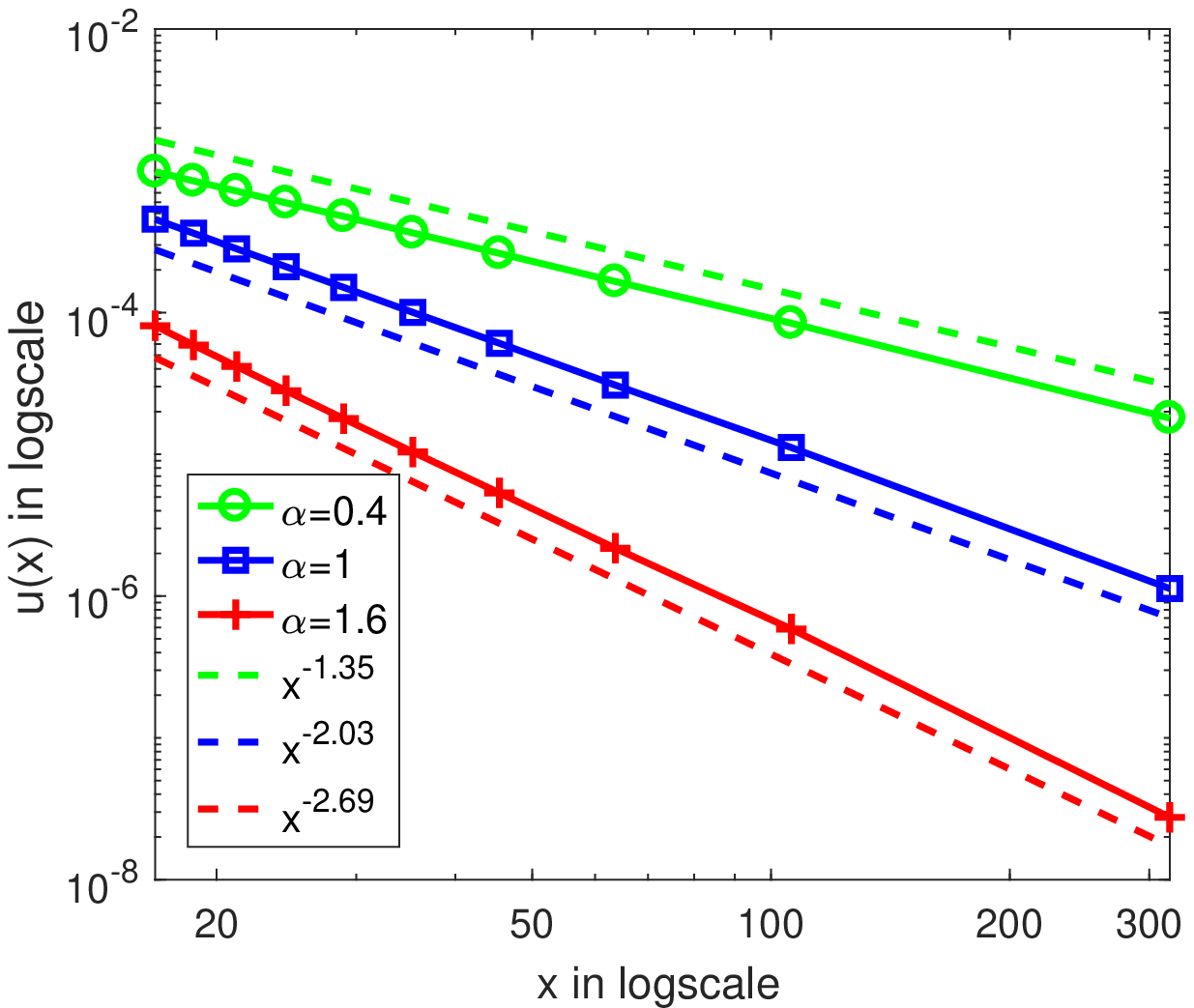}
\end{minipage}
\caption{Asymptotic behavior of $u(x)$ with different $\alpha$. Left: $f(x)=\exp(-{x^2}/{2})(1+x)$. Right: $f(x)=\frac{1}{(1+x^2)^2}$.}
\label{uchebyshev}
\end{figure}

\section{Modified rational spectral-collocation methods}\label{sect:6}
With the formulas in Theorem \ref{thm:mainformula} at our disposal, we can directly generate the spectral fractional differentiation matrices and develop the direct collocation methods, similar to the Hermite collocation methods  in \cite{hermitecollocation}.  However, it seems nontrivial and largely open to analyse its convergence.
In fact, we can also implement the collocation method in the Fourier transformed domain which turns to be more a natural way to extend the method to multiple dimensions.

\subsection{Fractional differentiation matrices}
Let $\{x_j^\lambda, \omega_j^\lambda\}_{j=0}^N$ be the mapped Gegenbauer-Gauss collocation points and weights as given in
\eqref{mappedxj}.  For any $u_N\in V_N^\lambda,$ we write
\begin{equation*}
u_{N}(x)=\sum_{j=0}^{N-1}u_{j}l_{j}(x), \quad \textmd{with} \quad l_{j}(x_k^\lambda) = \delta_{jk}, \quad 0\leq j,k\leq N-1,
\end{equation*}
where $u_j=u_N(x_j^\lambda).$ Note that the corresponding Lagrange basis function $\{l_{j}(x)\}_{j=1}^N$ can be expressed as
\begin{equation*}
l_j(x)=\sum_{k=0}^{N-1} \, b_{k}^{j}R_{k}^{\lambda}(x), \quad \textmd{with} \quad  b_{k}^{j}=\frac{R_{k}^{\lambda}(x_{j}^\lambda)\omega_{j}^\lambda} {\gamma_{k}^{\lambda}},  \quad 0\leq j,k\leq N-1.
\end{equation*}
Consequently,  we can easily derive the associated differential matrix $\mathcal{D}^{L,\alpha,\lambda}$ with Lagrange type bases
\begin{equation}\label{rationalcollocationmatrix}
\mathcal{D}_{i,j}^{L,\alpha,\lambda}=(-\Delta)^{\alpha/2}l_{j}(x_i^\lambda)=\sum_{k=0}^{N-1}b_{k}^{j}(-\Delta)^{\alpha/2}R_{j}^{\lambda}(x_{i}^\lambda),
\end{equation}
where $(-\Delta)^{\alpha/2}R_{j}^{\lambda}(x_{i}^\lambda)$ can be computed via \eqref{LapR2n} and \eqref{LapR2np1}.

\subsection{Numerical examples}
We now present several numerical examples to show the performance of the spectral collocation method based on MMGFs. Notice that the collocation method is more practical for problems with variable coefficients and nonlinear problems. Also, we shall carry out comparisons with the Hermite collocation method in \cite{hermitecollocation}.

\subsubsection{A multi-term fractional model}  We first consider the following multi-term fractional Laplacian equation:
\begin{equation}\label{rationalexample5.4}
\sum_{j=1}^{J}\rho_{j} (-\Delta)^{\alpha_{j}/2}u(x) = f(x), \quad x\in\mathbb{R}; \quad u(x)\to 0,\;\; {\rm as}\;\; |x|\to\infty.
\end{equation}
Here we set $J=4$ and
\begin{equation*}
\begin{split}
& \alpha_1= 0,  \quad \alpha_2= 0.5 \quad \alpha_3= 1.5, \quad \alpha_4= 2,\\
& \rho_1= \frac{\pi}{6},  \quad \rho_2= \frac{\pi}{3}, \quad \rho_3= \frac{\pi}{3}, \quad \rho_4= \frac{\pi}{6}.
\end{split}
\end{equation*}
Numerical results with two different souce terms are presented in Figure \ref{rationalmulti}. It can be seen that, similar to the Galerkin methods, the MMGF approach has a much better performance than the Hermite function approach in  all cases.
 \begin{figure}[!t]
\centering
\begin{minipage}[c]{0.45\textwidth}
\centering
\includegraphics[height=4.2cm,width=5.5cm]{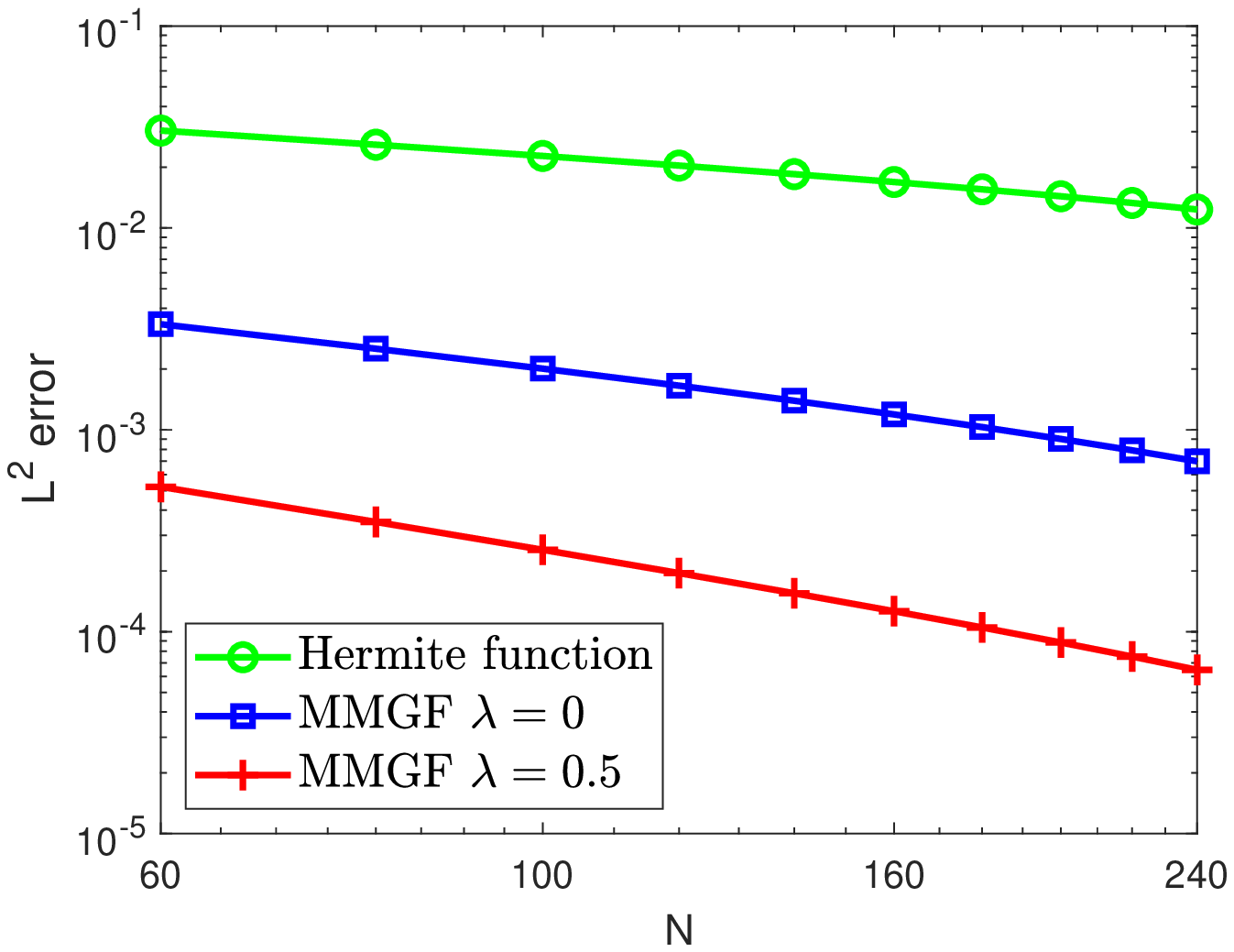}
\end{minipage}%
\begin{minipage}[c]{0.45\textwidth}
\centering
\includegraphics[height=4.2cm,width=5.5cm]{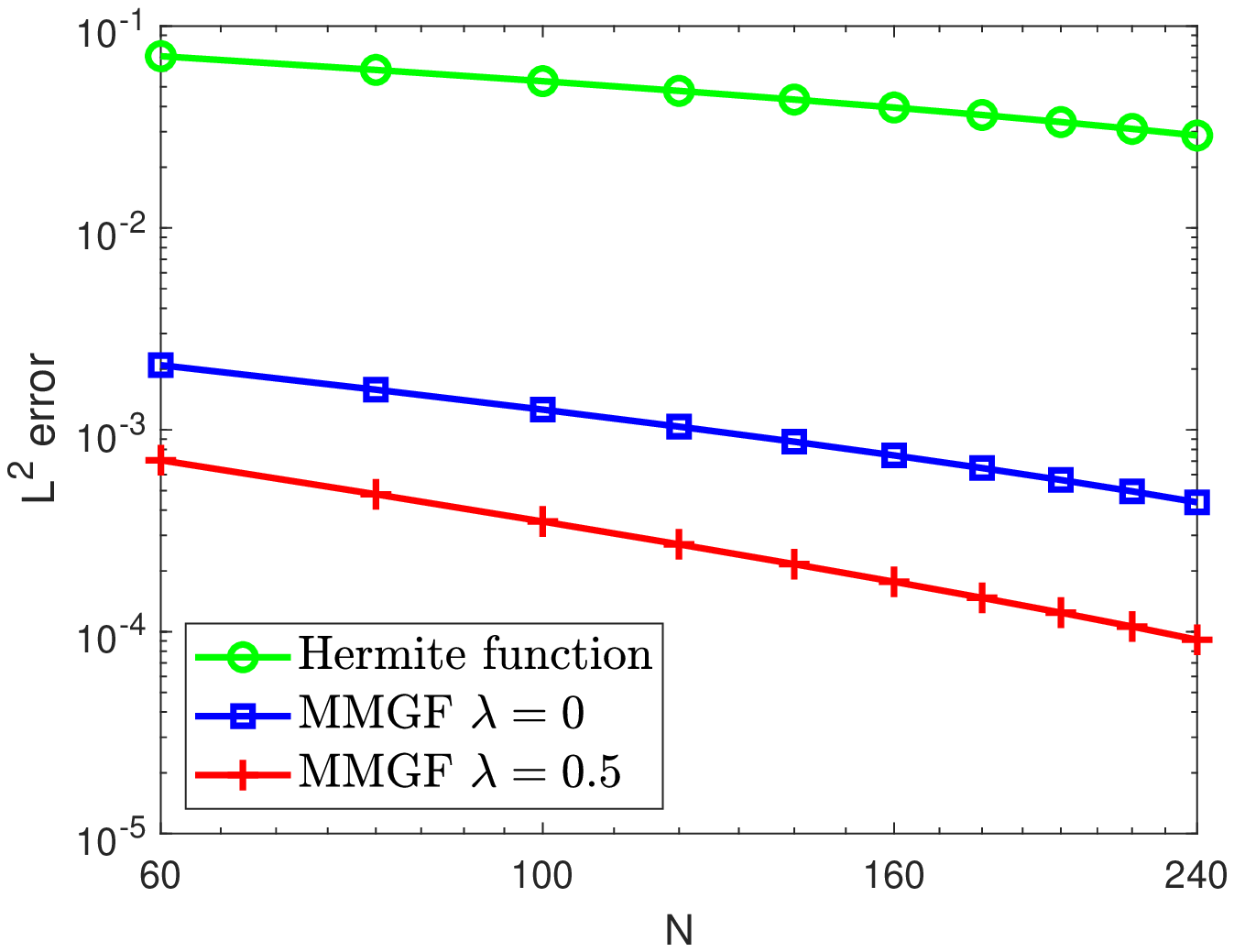}
\end{minipage}
\caption{Numerical results for the multi-term fractional model. Left: $f(x)=\exp(-\frac{x^2}{2})(1+x)$. Right: $f\left(x\right)=\dfrac{1}{\left(1+x^2\right)^{1.8}}$.}
\label{rationalmulti}
\end{figure}

\subsubsection{Fractional model with variable coefficients}
We next consider the following problem
\begin{equation}\label{rationalexample5.5}
\begin{split}
& (-\Delta)^{\alpha/2}u(x)+ r(x)u(x)=f(x),\quad x\in {\mathbb R},\\
&  u(x)\to 0,\;\; {\rm as}\;\; |x|\to\infty,
\end{split}
\end{equation}
where $r(x)=1+\exp(-x^2)$ and  $f(x)=\frac{1}{(1+x^2)^{1.2}}$. The convergence results for $\alpha=0.4,1,1.6$ are provided in Figure \ref{rationalnonlinear1} for both approaches. Again, the MMGFs spectral collocation method outperforms the Hermite collocation method.
\begin{figure}[!t]
\centering
\begin{minipage}[c]{0.33\textwidth}
\centering
\includegraphics[height=4.2cm,width=4.3cm]{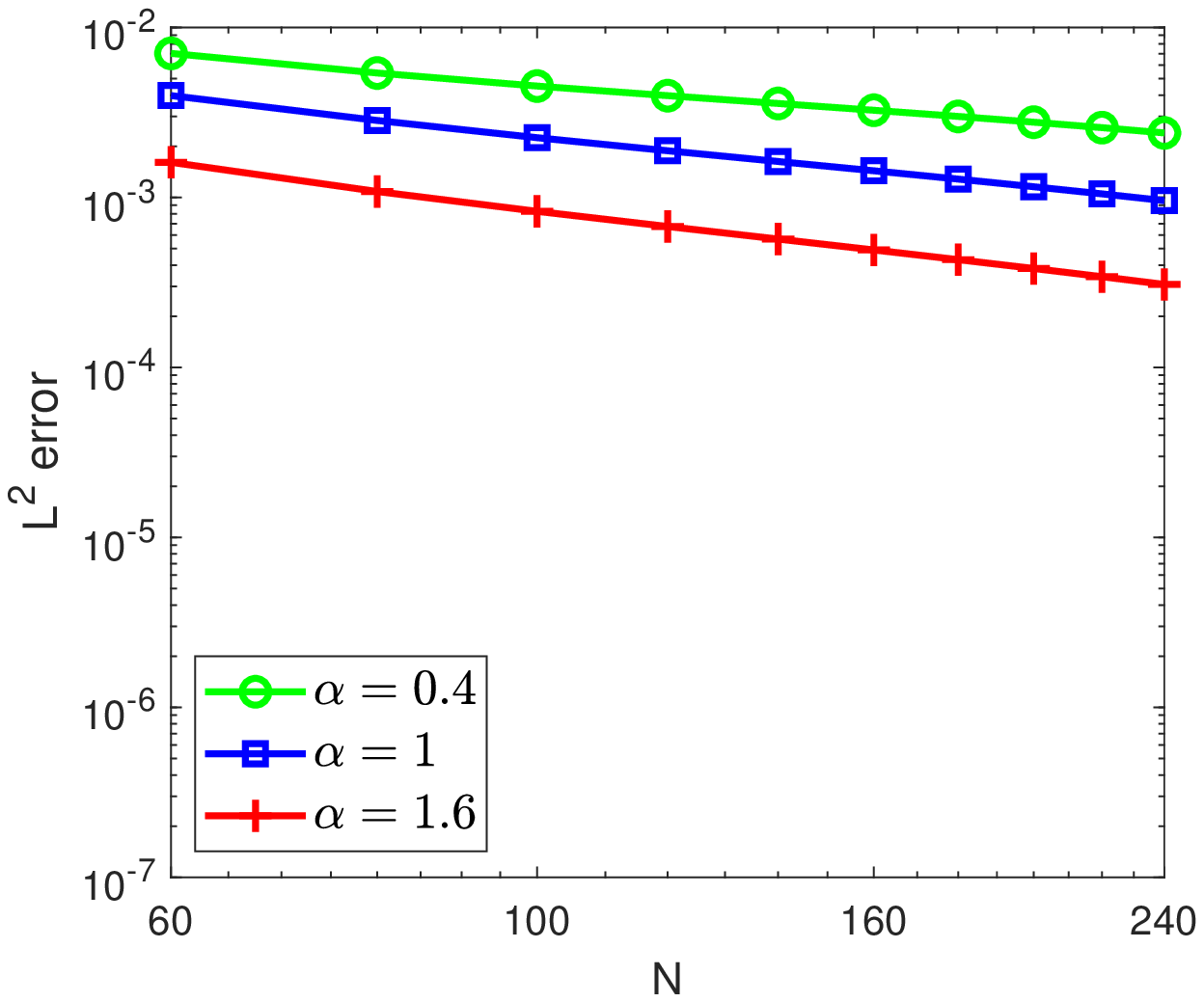}
\end{minipage}%
\begin{minipage}[c]{0.33\textwidth}
\centering
\includegraphics[height=4.2cm,width=4.3cm]{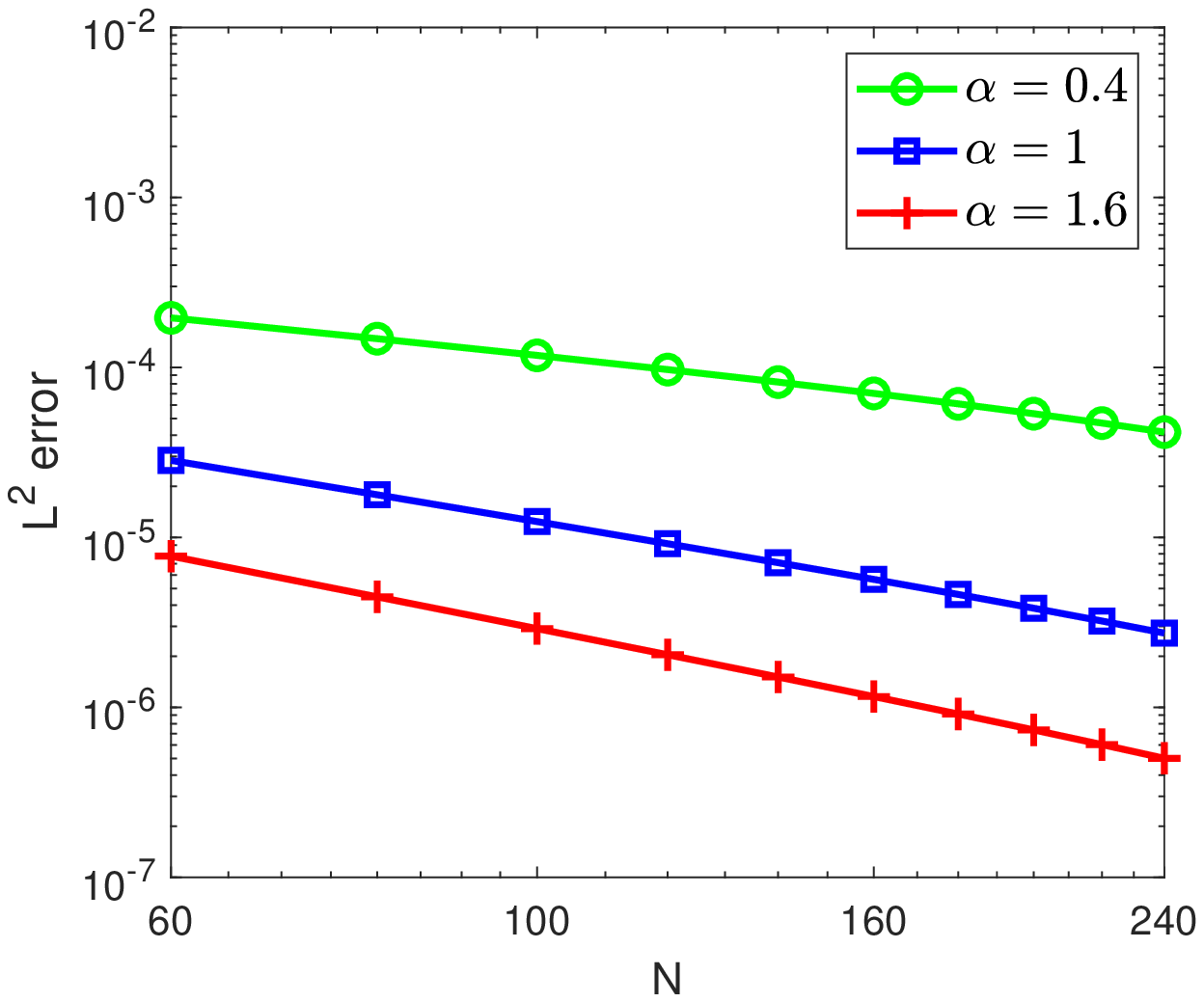}
\end{minipage}
\begin{minipage}[c]{0.33\textwidth}
\centering
\includegraphics[height=4.2cm,width=4.3cm]{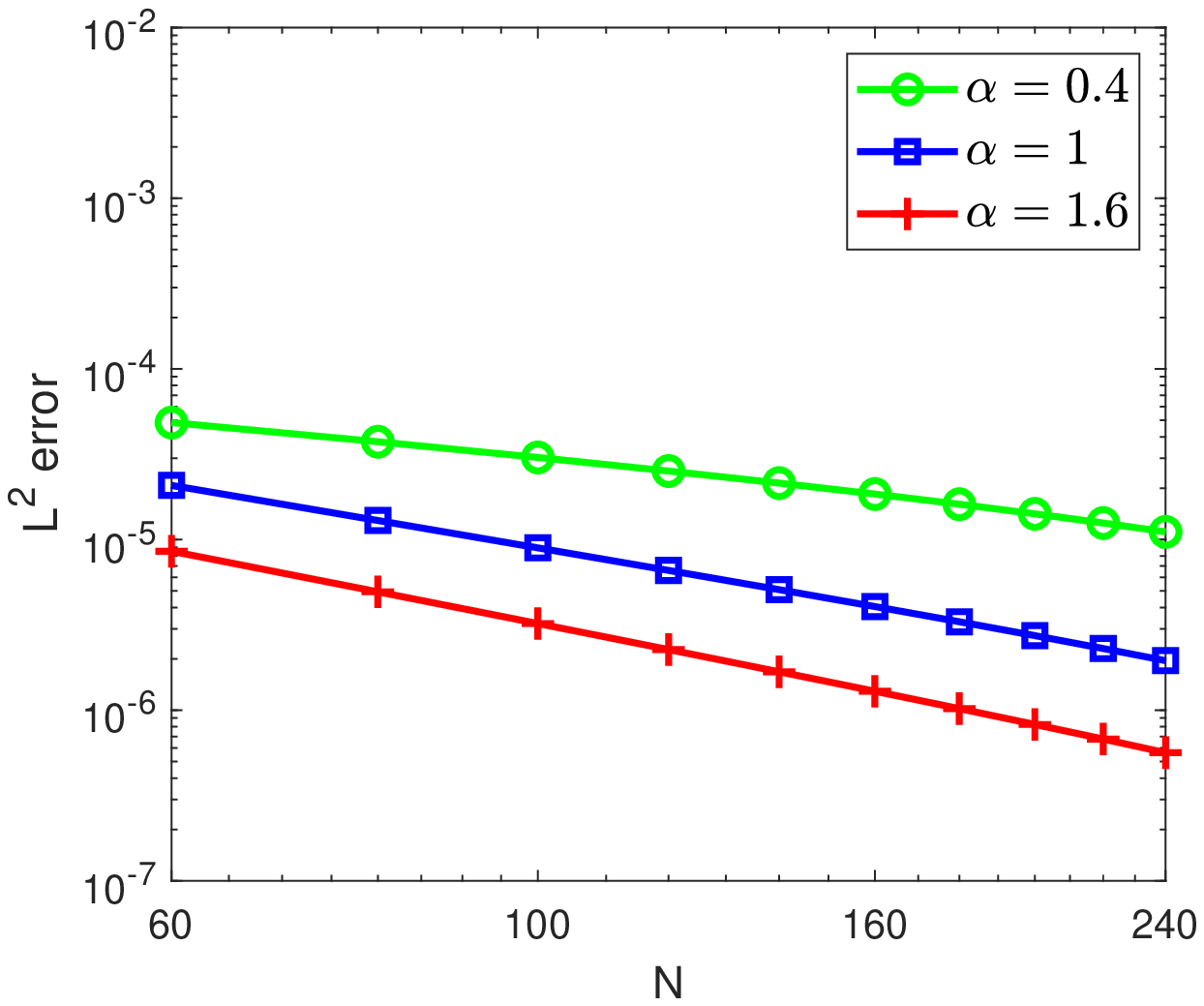}
\end{minipage}
\caption{Approximation error for equation \eqref{rationalexample5.5} with $f(x)=\frac{1}{(1+x^2)^{1.2}}$. Left: Hermite collocation methods in \cite{hermitecollocation} with scaling factor $\mu=2.5.$ Middle: The MMGFs collocation method  with $\lambda=0$ and scaling factor $\mu=5.$  Right: The MMGFs collocation method with $\lambda=0.5$ and scaling factor $\mu=5.$}
\label{rationalnonlinear1}
\end{figure}

\subsubsection{An eigenvalue problem} Finally,  we consider the following eigenvalue problem as in \cite{hermitecollocation}:
\begin{equation}\label{rationalexample 5.6}
((-\Delta)^{\alpha/2}+x^2)u(x)=\lambda u(x),\quad x\in {\mathbb R}.
\end{equation}
Notice that exact eigenvalues for the case of $\alpha=1$ are available in \cite{eigenvalue}. For this example, we shall compute the first three eigenvalues by the MMGFs spectral collocation method and the Hermite collocation methods for comparison. The numerical results are given in Figure \ref{eigenvalue}, which  shows that
 the MMGF collocation method is more accurate than the Hermite collocation method.

\begin{figure}
\centering
\begin{minipage}[c]{0.31\textwidth}
\centering
\includegraphics[height=4.0cm,width=4.3cm]{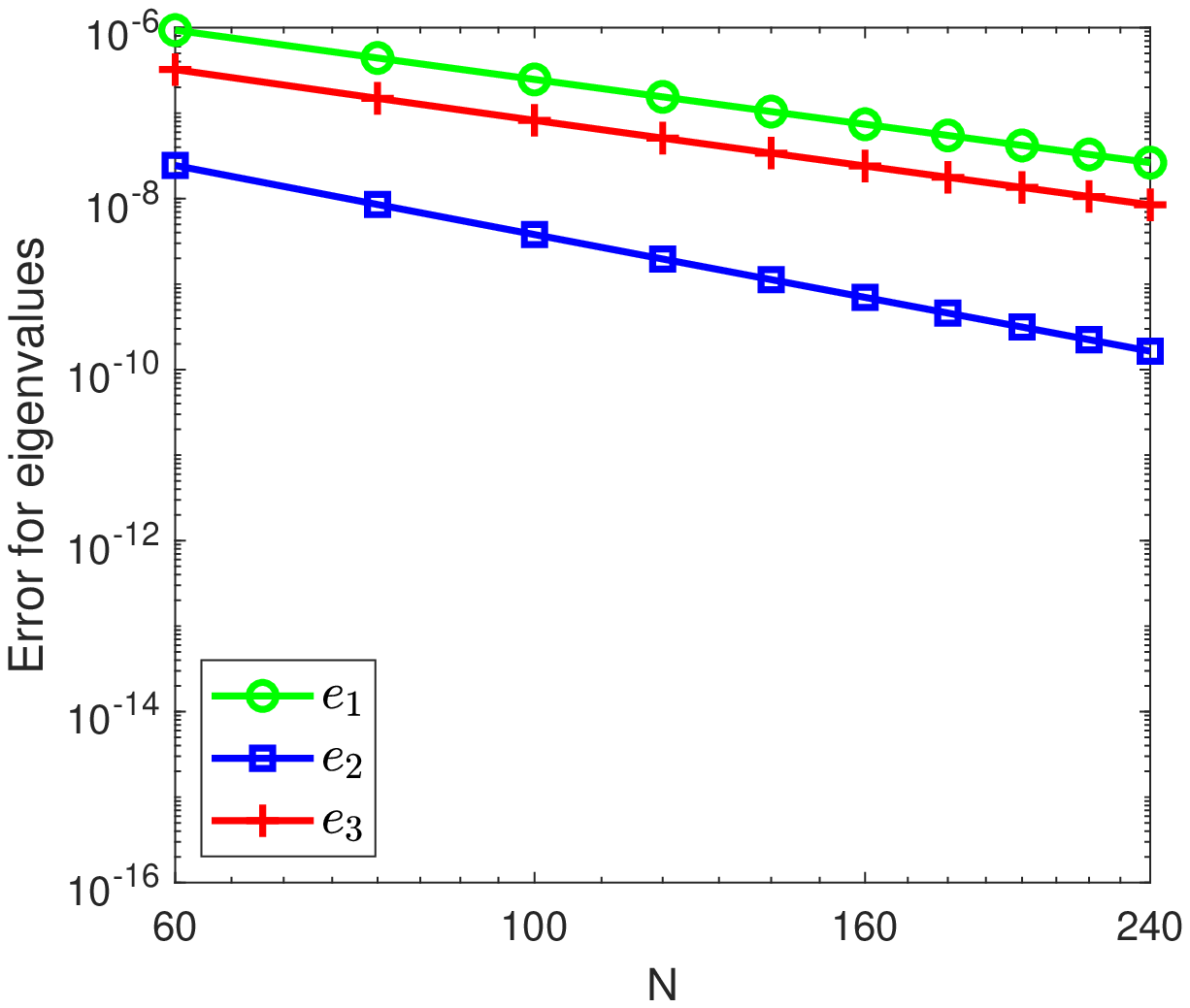}
\end{minipage}%
\begin{minipage}[c]{0.31\textwidth}
\centering
\includegraphics[height=4.0cm,width=4.3cm]{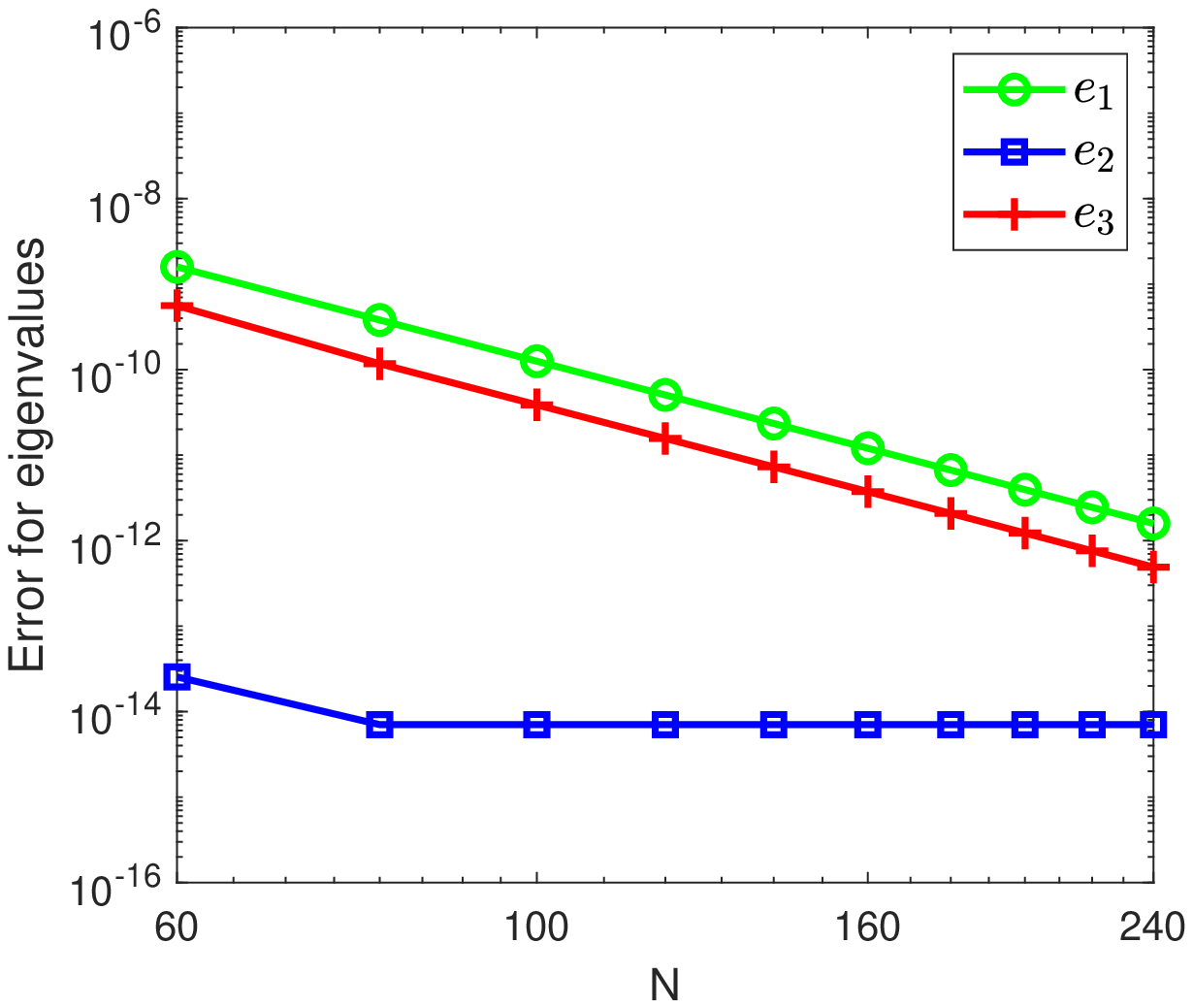}
\end{minipage}
\begin{minipage}[c]{0.31\textwidth}
\centering
\includegraphics[height=4.0cm,width=4.3cm]{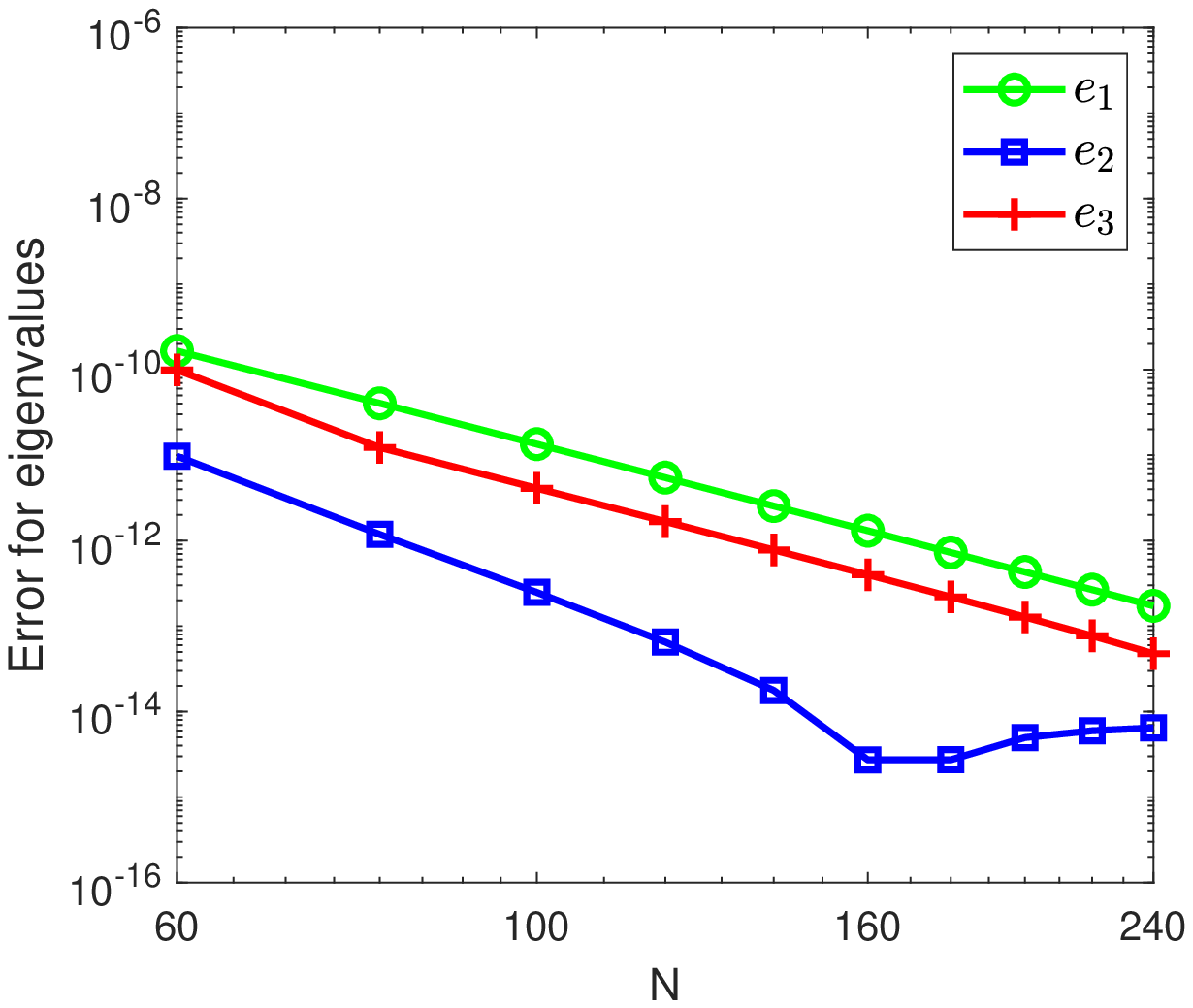}
\end{minipage}
\caption{Numerical error for the first three eigenvalues of \eqref{rationalexample 5.6}. Left: Generalized Hermite function. Middle: MMGFs with $\lambda=0$. Right: MMGFs with $\lambda=0.5$.}
\label{eigenvalue}
\end{figure}

\subsection{Spectral-collocation methods in  multiple  dimensions} To this end, we propose the modified rational collocation methods based on a formulation in the Fourier transformed domain  in multiple dimensions and show that it is more accurate than the Hermite spectral collocation methods in \cite{MaoShen}.

To fix the idea, we  consider the $d$-dimensional model problem:
\begin{equation}\label{multFractional}
\begin{split}
 &  (-\Delta)^{\alpha/2}u(x) +  \rho u(x) = f(x), \quad x\in  \mathbb{R}^d;\\
 & u(x)\to 0,\quad {\rm as}\;\; | x|\to \infty,
  \end{split}
\end{equation}
where we denote $ x=(x_1,\cdots, x_d)$ and $| x|=\sqrt{ x^t  x}.$
It is known that in the Fourier transformed domain, it can be expressed as
\begin{equation}\label{domaintrun}
    (|\xi|^{\alpha} +\rho)\hat{u}(\xi) = \hat{f}( \xi), \quad  \xi\in {\mathbb R}^d,
\end{equation}
where $\hat u,\hat f$ are the Fourier transform of $u,f,$ respectively.  Thus, we have
\begin{equation}
   \hat{u}( \xi)=\frac{ \hat{f}( \xi)}{| \xi|^{\alpha} +\rho}, \quad  \xi\in {\mathbb R}^d.
\end{equation}
That is, the Fourier transform of solution $u$ can be expressed explicitly as above.  This motivates the construction of the collocation method in the frequency space. To describe the algorithm, we denote
$$
\Upsilon_{\!N}:=\big\{ j=(j_1,\cdots, j_d)\,:\, j_i=0,1,\cdots,N,\;\; 1\le i\le d\big\},
$$
and define the tensorial grids and tensorial MMGFs as
\begin{equation}\label{bsxK}
 x_{ j}^{\lambda}=(x_{j_1}^\lambda, \cdots, x_{j_d}^\lambda),\quad  j\in \Upsilon_{\!N}; \quad
 R_{ n}^\lambda( x)=\prod_{i=1}^d R_{n_i}^\lambda(x_i).
\end{equation}

As the first step, we approximate $f( x)$ by the multidimensional interpolation:
\begin{equation}\label{tensor-intp}
{ I}_{N}^\lambda f( x)=\sum_{ n\in \Upsilon_{\!N}}\tilde{f}_{ n}\, R_{ n}^\lambda( x),
\end{equation}
where the coefficients $\{\tilde{f}_{ n}\}_{ n\in \Upsilon_{\!N}}$ can be computed from the samples $\{f( x_{ j}^\lambda)\}_{
 j\in \Upsilon_{\!N}}$ by the tensorial version of the quadrature \eqref{gujacobi}.
Then we have the approximation:
\begin{equation}\label{tensor-intps}
\hat{f}( \xi)\approx \widehat{{ I}_{N}^\lambda f}( \xi)=\sum_{ n\in \Upsilon_{\!N}}\tilde{f}_{ n}\, \prod_{i=1}^d {\mathscr F}[R_{n_i}^\lambda](\xi_i),
\end{equation}
where ${\mathscr F}[R_{n_i}^\lambda](\xi_i)$ can be computed by the formulas in Theorem \ref{thm:mainformulaAA} below.

Then the direct collocation approximation of $u( x)$ in the frequency space is given by
\begin{equation}\label{2dcollocation}
   \widehat{u}_{N}( \xi_{ j}^{\lambda})=\frac{ \widehat{{ I}_{N}^\lambda f}(
    \xi_{ j}^{\lambda})}{| \xi_{ j}^{\lambda}|^2+\rho},\quad  j\in \Upsilon_{\!N},
\end{equation}
where $\{ \xi_{ j}^{\lambda}\}$ are the tensorial grids as in \eqref{bsxK}.  With these samples, we can write the final approximation $u_N( x)$ by taking Fourier inverse transform of $\hat u_N( \xi)$ as follows:
\begin{equation}\label{funspAs}
u_N( x)=\sum_{ n\in \Upsilon_{\!N}}\tilde{u}_{ n}\, \prod_{i=1}^d {\mathscr F}^{-1}[R_{n_i}^\lambda](x_i),
\end{equation}
where the coefficients $\{\tilde{u}_{ n}\}_{ n\in \Upsilon_{\!N}}$ can be computed from
$\{\widehat{u}_{N}( \xi_{ j}^{\lambda})\}_{
 j\in \Upsilon_{\!N}}$ in \eqref{2dcollocation} by the quadrature formula (cf. \eqref{gujacobi}) as before.
Here, the inverse Fourier transforms can be computed by the formulas in Theorem \ref{thm:mainformulaAA} and Remark \ref{newABC} below.

Like Theorem \ref{thm:mainformula}, we have the following formulas for computing the Fourier transform of the MMGFs.
\begin{theorem}\label{thm:mainformulaAA} For real $\lambda>-1/2,$
the Fourier transform of the MMGFs can be computed by
\begin{equation}\label{LapR2nAA}
\begin{split}
 {\mathscr F}[R_{2n}^\lambda](\xi)& =  
   a_n^\lambda  
 \sum_{k=0}^{n}\frac{(-n)_{k}(n+\lambda)_{k}}
{(\lambda+\frac{1}{2})_{k}\, k!}  \frac{|\xi|^{k+\lambda/2}K_{k+\lambda/2}(|\xi|)} {2^{k+(\lambda-1)/2}\Gamma(k+(\lambda+1)/2)},
 \end{split}
\end{equation}
and
\begin{equation}\label{LapR2np1AA}
\begin{split}
 {\mathscr F}[R_{2n+1}^\lambda](\xi)&=  -{\rm i} \, {\rm sign}(\xi) \,b_n^\lambda  
 \sum_{k=0}^{n}\frac{(-n)_{k}(n+\lambda+1)_{k}}
{(\lambda+\frac{1}{2})_{k}k!} \,
 \frac{|\xi|^{k+\lambda/2}K_{k+(\lambda-1)/2}(|\xi|)} {2^{k+\lambda/2}\Gamma(k+1+\lambda/2)},
 \end{split}
\end{equation}
where the constants $a_n^\lambda, b_n^\lambda$  are defined in \eqref{anbnlam}.
\end{theorem}
\begin{proof} By \eqref{neweqnA}, we have that for $\gamma>0,$
$$
{\mathscr F}\Big[\frac 1 {(1+x^2)^\gamma}\Big](\xi) =\dfrac{2^{1-\gamma}} {\Gamma(\gamma)} |\xi|^{\gamma-\frac{1}{2}}{K_{\gamma-\frac{1}{2}}(|\xi|)},\quad \xi\in {\mathbb R}.
$$
Similarly, we derive from  \eqref{neweqnA0} that for $\gamma>1/2,$
$$
{\mathscr F}\Big[\frac x {(1+x^2)^\gamma}\Big](\xi)= -{\rm i}  \, \dfrac{2^{1-\gamma}} {\Gamma(\gamma)}\, {\rm sign}(\xi) \, |\xi|^{\gamma-\frac{1}{2}}{K_{\gamma-\frac{3}{2}}(|\xi|)},\quad \xi\in {\mathbb R}.
$$
Consequently, the formulas \eqref{LapR2nAA}-\eqref{LapR2np1AA} follow from \eqref{Rnlambda1}-\eqref{Rnlambda2}
directly.
\end{proof}
\begin{remark}\label{newABC} In \eqref{funspAs}, we need the inverse transform of $R_{n}^\lambda(\xi),$ which can be computed by the same formulas \eqref{Rnlambda1}-\eqref{Rnlambda2}. Indeed, by definition, we have
\begin{equation}\label{inverse}
{\mathscr F}^{-1}[R_{n}^\lambda](x)=\frac 1 {\sqrt{2\pi}}
\int_{-\infty}^{\infty} e^{{\rm i
} x\xi} R_{n}^\lambda(\xi)d\xi= \overline{{\mathscr F}[R_{n}^\lambda](x)}.
\end{equation}
\end{remark}

%

We now consider a two dimensional example with $f(x,y)=\exp(-\sqrt{x^2+y^2})$. Notice that the Fourier transform of this source term can be computed as
\begin{align*}
 \mathscr{F}[f](\xi,\eta)=\frac{1}{\left(1+\xi^2+\eta^2\right)^{3/2}}.
\end{align*}
The corresponding numerical results are presented in Figure \ref{rational2d1}. Once again,
the MMGF collocation method is more accurate and  converges faster than the Hermite collocation method.
 \begin{figure}
\centering
\begin{minipage}[c]{0.31\textwidth}
\centering
\includegraphics[height=4.2cm,width=4.3cm]{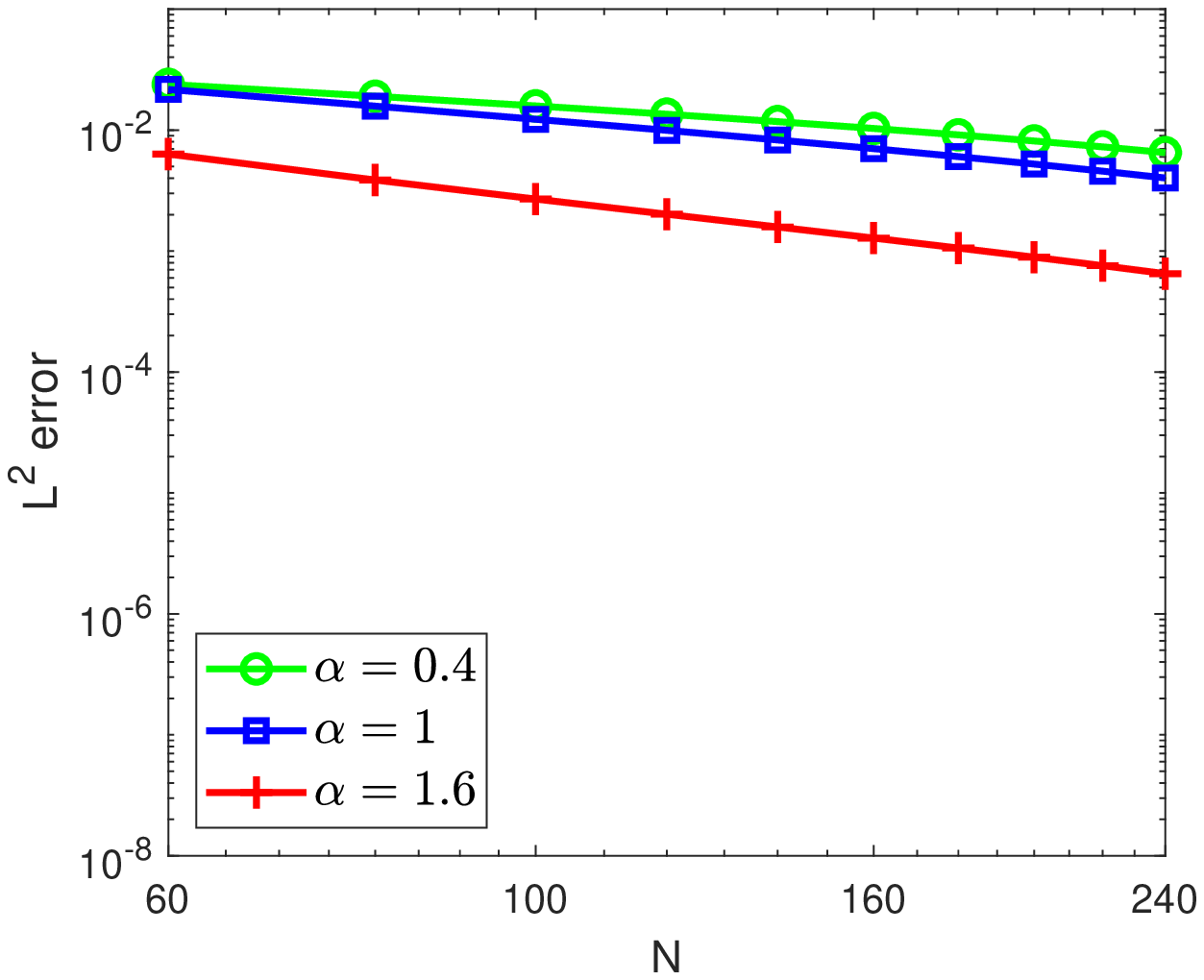}
\end{minipage}%
\begin{minipage}[c]{0.31\textwidth}
\centering
\includegraphics[height=4.2cm,width=4.3cm]{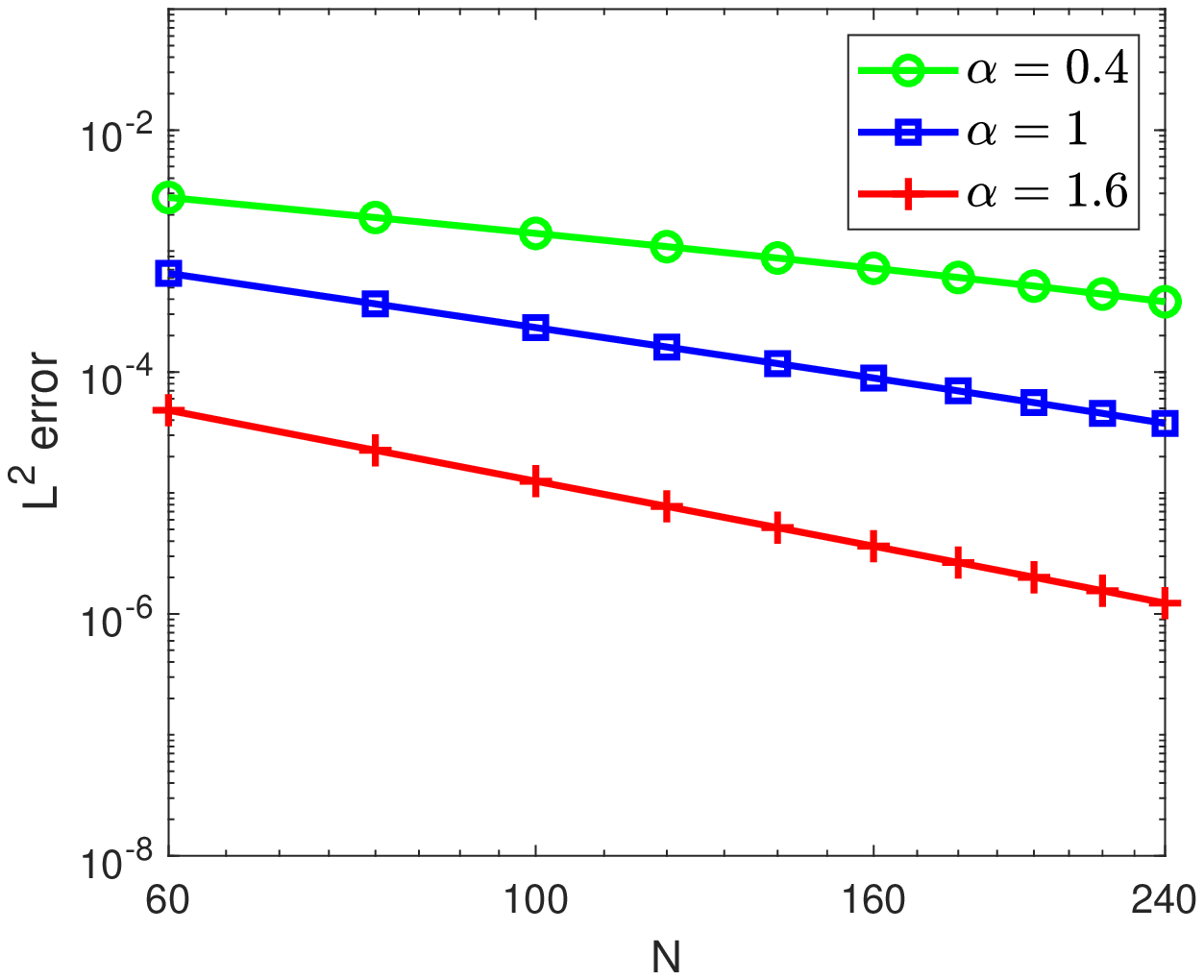}
\end{minipage}
\begin{minipage}[c]{0.31\textwidth}
\centering
\includegraphics[height=4.2cm,width=4.3cm]{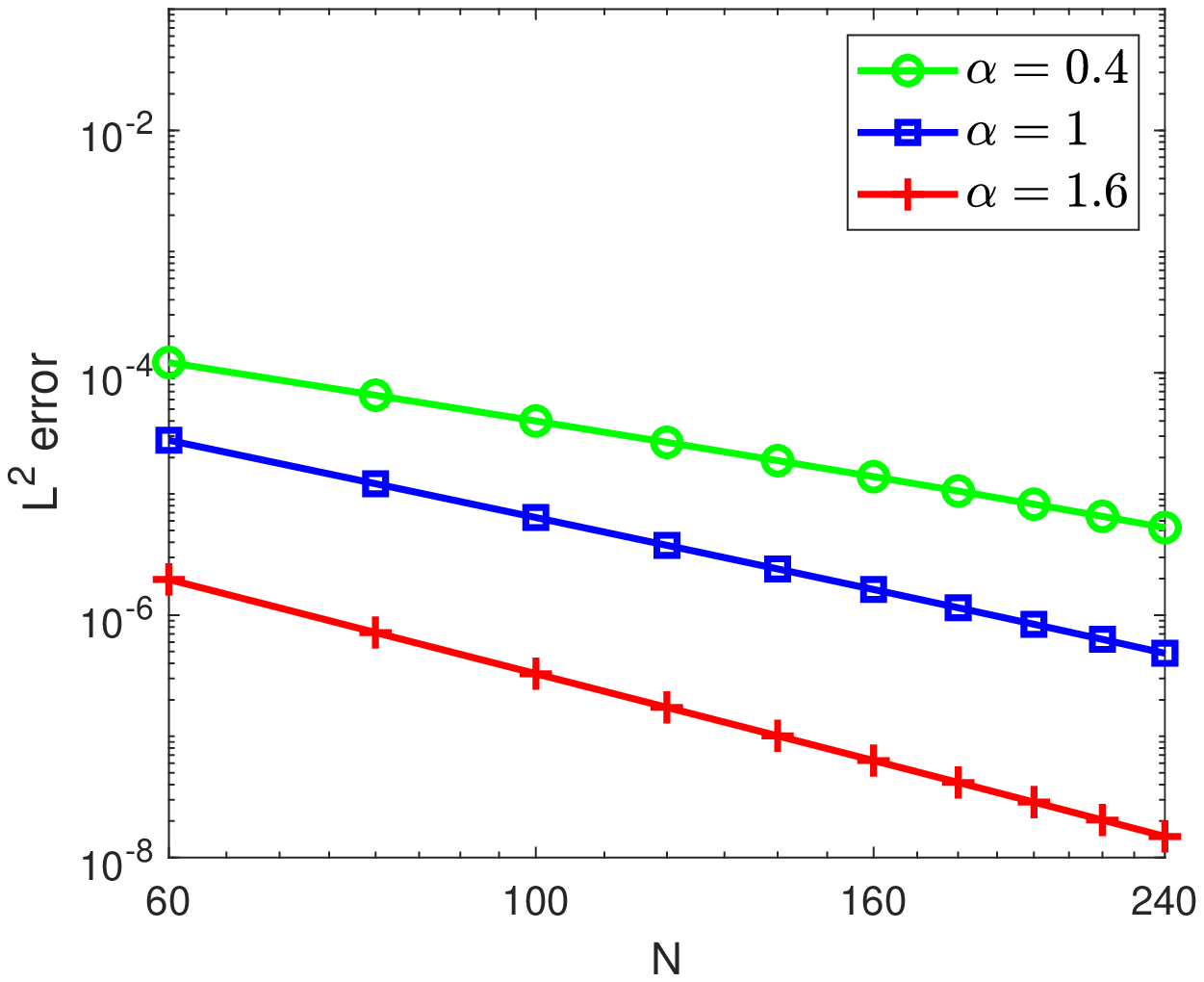}
\end{minipage}
\caption{Numerical results for the two dimensional example with $f(x,y)=\exp(-\sqrt{x^2+y^2})$. Left: Hermite collocation methods. Middle: MMGF collocation methods with $\lambda=0$. Right: MMGF collocation methods with $\lambda=0.5$.}
\label{rational2d1}
\end{figure}

\section{Summary and concluding remarks}
\label{sec7}
In this paper, we have developed accurate spectral methods using rational basis (or modified mapped Gegenbauer functions) for  PDEs with fractional Laplacian in unbounded domains. The main building block of the spectral algorithms is some explicit formulas for the Fourier transforms and fractional Laplacian of the rational basis. With these, we can construct rational spectral-Galerkin and collocation schemes by pre-computing the associated fractional differentiation matrices. We obtain optimal error estimates of rational spectral approximation in the fractional Sobolev spaces, and analyze the optimal convergence of the proposed Galerkin scheme. Numerical results show that the rational method outperforms the Hermite function approach. Future studies along this line include the error estimates of the rational collocation methods in section 6, fast pre-conditioner/solvers for high dimensional problems, and applications of the MMGFs approach to tempered fractional PDEs.

\bibliographystyle{siam}
\bibliography{refrational}

\begin{thebibliography}{10}

\bibitem{Agranovich2015Book}
{\sc M.S. Agranovich}, {\em Sobolev spaces, their generalizations and elliptic
  problems in smooth and Lipschitz domains}, Springer, 2015.

\bibitem{Andrews99}
{\sc G.E. Andrews, R.~Askey, and R.~Roy}, {\em Special functions (encyclopedia
  of mathematics and its applications vol 71)}, 1999.

\bibitem{BJ17}
{\sc D.H. Baffet and J.S. Hesthaven}, {\em A kernel compression scheme for
  fractional differential equations}, SIAM J. Numer. Anal., 55 (2017),
  pp.~496--520.

\bibitem{physics_II}
{\sc E.~Barkai, R.~Metzler, and J.~Klafter}, {\em From continuous time random
  walks to the fractional {Fokker-Planck} equation}, Phys. Rev. E, 61 (2000),
  p.~132.

\bibitem{BFW98}
{\sc P.~Biler, T.~Funaki, and W.A. Woyczynski}, {\em Fractal {B}urgers
  equations}, J. Diff. Equat., 148 (1998), pp.~9--46.

\bibitem{BBNO2018}
{\sc A.~Bonito, J.P. Borthagaray, R.H. Nochetto, E.~Ot\'{a}rola, and A.J.
  Salgado}, {\em Numerical methods for fractional diffusion}, Comput. Vis.
  Sci., 19 (2018), pp.~19--46.

\bibitem{Boyd87a}
{\sc J.P. Boyd}, {\em Spectral methods using rational basis functions on an
  infinite interval}, J. Comput. Phys., 69 (1987), pp.~112--142.

\bibitem{Boyd01}
{\sc J.P. Boyd}, {\em {Chebyshev and {F}ourier Spectral Methods}}, Dover
  Publications Inc., Mineola, NY, second~ed., 2001.

\bibitem{CX13}
{\sc J.Y. Cao and C.J. Xu}, {\em A high order schema for the numerical solution
  of the fractional ordinary differential equations}, J. Comput. Phys., 238
  (2013), pp.~154--168.

\bibitem{Chen_Mao_Shen}
{\sc S.~Chen, J.~Shen, and Z.~Mao}, {\em Efficient and accurate spectral
  methods using general {J}acobi functions for solving {R}iesz fractional
  differential equations}, Appl. Numer. Math, 106 (2016), pp.~165--181.

\bibitem{spectralshen}
{\sc S.~Chen, J.~Shen, and L.L. Wang}, {\em Generalized {J}acobi functions and
  their applications to fractional differential equations}, Math. Comp., 85
  (2016), pp.~1603--1638.

\bibitem{CSW.18}
\leavevmode\vrule height 2pt depth -1.6pt width 23pt, {\em Laguerre functions
  and their applications to tempered fractional differential equations on
  infinite intervals}, J. Sci. Comput., 74 (2018), pp.~1286--1313.

\bibitem{dPQRV11}
{\sc A.~de~Pablo, F.~Quir{\'o}s, A.~Rodr{\'\i}guez, and J.L. V{\'a}zquez}, {\em
  A fractional porous medium equation}, Adv. Math., 226 (2011), pp.~1378--1409.

\bibitem{FGT02}
{\sc J.C.M. Fok, B.Y. Guo, and T.~Tang}, {\em Combined {H}ermite
  spectral-finite difference method for the {F}okker-{P}lanck equation}, Math.
  Comp., 71 (2002), pp.~1497--1528 (electronic).

\bibitem{tableofintegrals}
{\sc I.S. Gradshteyn and I.M. Ryzhik}, {\em Table of integrals, series, and
  products, translated from the russian, translation edited and with a preface
  by daniel zwillinger and victor moll, revised from the seventh edition},
  2015.

\bibitem{GSW00}
{\sc B.Y. Guo, J.~Shen, and Z.Q. Wang}, {\em A rational approximation and its
  applications to differential equations on the half line}, J. Sci. Comp., 15
  (2000), pp.~117--147.

\bibitem{modifiedchebyshev}
{\sc B.Y. Guo and Z.Q. Wang}, {\em Modified {C}hebyshev rational spectral
  method for the whole line}, in Proceedings of the fourth international
  conference on dynamical systems and differential equations, 2002,
  pp.~365--374.

\bibitem{differenceHuang}
{\sc Y.H Huang and A.~Oberman}, {\em Numerical methods for the fractional
  {L}aplacian: A finite difference-quadrature approach}, SIAM J. Numer. Anal.,
  52 (2014), pp.~3056--3084.

\bibitem{differenceJi}
{\sc C.C. Ji and Z.Z. Sun}, {\em A high-order compact finite difference scheme
  for the fractional sub-diffusion equation}, J. Sci. Comput., 64 (2015),
  pp.~959--985.

\bibitem{elementJin}
{\sc B.T. Jin, R.~Lazarov, and Z.~Zhou}, {\em Error estimates for a
  semidiscrete finite element method for fractional order parabolic equations},
  SIAM J. Numer. Anal., 51 (2013), pp.~445--466.

\bibitem{KZK16}
{\sc E.~Kharazmi, M.~Zayernouri, and G.E. Karniadakis}, {\em Petrov--{G}alerkin
  and spectral collocation methods for distributed order differential
  equations}, SIAM J. Sci. Comput., 39 (2017), pp.~A1003--A1037.

\bibitem{spectralArab}
{\sc H.~Khosravian-Arab, M.~Dehghan, and M.R. Eslahchi}, {\em Fractional
  {S}turm--{L}iouville boundary value problems in unbounded domains: Theory and
  applications}, J. Comput. Phys., 299 (2015), pp.~526--560.

\bibitem{Lan72}
{\sc N.S. Landkof}, {\em Foundations of modern potential theory springer}, New
  York,  (1972).

\bibitem{Lions1972Book}
{\sc J.L. Lions and E.~Magenes}, {\em Non-homogeneous boundary value problems
  and applications}, Vol. II. Die Grundlehren der mathematischen
  Wissenschaften. Springer-Verlag, New York-Heidelberg,  (1972).

\bibitem{LZK16}
{\sc A.~Lischke, M.~Zayernouri, and G.E. Karniadakis}, {\em A tunably-accurate
  laguerre petrov-galerkin spectral method for multi-term fractional
  differential equations on the half line}, arXiv preprint arXiv:1607.08579,
  (2016).

\bibitem{eigenvalue}
{\sc J.~Lorinczi and J.~Malecki}, {\em Spectral properties of the massless
  relativistic harmonic oscillator}, arXiv preprint arXiv:1006.3665,  (2010).

\bibitem{Ma.ST05}
{\sc H.P. Ma, W.W. Sun, and T.~Tang}, {\em Hermite spectral methods with a
  time-dependent scaling for parabolic equations in unbounded domains}, SIAM J.
  Numer. Anal., 43 (2005), pp.~58--75.

\bibitem{MaoShen}
{\sc Z.P. Mao and J.~Shen}, {\em Hermite spectral methods for fractional {PDEs}
  in unbounded domains}, SIAM J. Sci. Comput., 39 (2017), pp.~A1928--A1950.

\bibitem{Olver2010Book}
{\sc F.W.J. Olver, D.W. Lozier, R.F. Boisvert, and C.W. Clark}, {\em {NIST}
  handbook of mathematical functions cambridge university press}, New York,
  (2010).

\bibitem{RYW19}
{\sc Y.~Ren, X.~Yu, and Z.Q. Wang}, {\em Diagonalized {C}hebyshev rational
  spectral methods for second-order elliptic problems on unbounded domains},
  Numer. Math. Theor. Meth. Appl, 12 (2019), pp.~265--284.

\bibitem{SWT}
{\sc J.~Shen, T.~Tang, and L.L. Wang}, {\em Spectral methods: algorithms,
  analysis and applications}, vol.~41, Springer, Berlin, 2011.

\bibitem{She.W09}
{\sc J.~Shen and L.L. Wang}, {\em Some recent advances on spectral methods for
  unbounded domains}, Commun. Comp. Phys., 5 (2009), pp.~195--241.

\bibitem{SWY14}
{\sc J.~Shen, L.L. Wang, and H.J. Yu}, {\em Approximations by orthonormal
  mapped {C}hebyshev functions for higher-dimensional problems in unbounded
  domains}, J. Comput. Appl. Math., 265 (2014), pp.~264--275.

\bibitem{SS17}
{\sc C.T. Sheng and J.~Shen}, {\em A hybrid spectral element method for
  fractional two-point boundary value problems}, Numer. Math. Theor. Meth.
  Appl., 10 (2017), pp.~437--464.

\bibitem{SS18}
\leavevmode\vrule height 2pt depth -1.6pt width 23pt, {\em A space-time
  {P}etrov-{G}alerkin spectral method for time fractional diffusion equation},
  Numer. Math. Theor. Meth. Appl, 11 (2018), pp.~854--876.

\bibitem{SX17}
{\sc F.Y. Song, C.J. Xu, and G.E. Karniadakis}, {\em Computing fractional
  {L}aplacians on complex-geometry domains: algorithms and simulations}, SIAM
  J. Sci. Comput., 39 (2017), pp.~A1320--A1344.

\bibitem{SLW18}
{\sc T.~Sun, R.~Liu, and L.L. Wang}, {\em Generalised {M}\"untz spectral
  {G}alerkin methods for singularly perturbed fractional differential
  equations}, East. Asia. J. Appl. Math., 8 (2018), pp.~611--633.

\bibitem{Tang93}
{\sc T.~Tang}, {\em The {H}ermite spectral method for {G}aussian-type
  functions}, SIAM J. Sci. Comput., 14 (1993), pp.~594--606.

\bibitem{hermitecollocation}
{\sc T.~Tang, H.F. Yuan, and T.~Zhou}, {\em Hermite spectral collocation
  methods for fractional {PDEs} in unbounded domains}, Commun. Comput. Phys.,
  24 (2018), pp.~1143--1168.

\bibitem{TD15}
{\sc W.Y. Tian, H.~Zhou, and W.H. Deng}, {\em A class of second order
  difference approximations for solving space fractional diffusion equations},
  Math. Comp., 84 (2015), pp.~1703--1727.

\bibitem{TYZ18}
{\sc T.Tang, H.~Yu, and T.~Zhou}, {\em On energy dissipation theory and
  numerical stability for time-fractional phase field equations},
  arXiv:1808.01471,  (2018).

\bibitem{modifiedlegendre}
{\sc Z.Q. Wang and B.Y Guo}, {\em Modified {L}egendre rational spectral method
  for the whole line}, J. Comput. Math.,  (2004), pp.~457--474.

\bibitem{WZ17}
{\sc S.L Wu and T.~Zhou}, {\em Fast parareal iterations for fractional
  diffusion equations}, J. Comput. Phys., 329 (2017), pp.~210--226.

\bibitem{WZ18}
\leavevmode\vrule height 2pt depth -1.6pt width 23pt, {\em Parareal algorithms
  with local time-integrators for time fractional differential equations}, J.
  Comput. Phys., 385 (2018), pp.~135--149.

\bibitem{GLW18}
{\sc X.Guo, Y.~Li, and H.~Wang}, {\em A fast finite difference method for
  tempered fractional diffusion equations}, Commun. Comput. Phys., 24 (2018),
  pp.~531--556.

\bibitem{YM18}
{\sc Y.Yang and H.Ma}, {\em The legendre {G}alerkin-{C}hebyshev collocation
  method for space fractional {B}urgers-like equations}, Numer. Math. Theor.
  Meth. Appl, 11 (2018), pp.~338--353.

\bibitem{spectralgeorge}
{\sc M.~Zayernouri and G.E. Karniadakis}, {\em Fractional {S}turm--{L}iouville
  eigen-problems: theory and numerical approximation}, J. Comput. Phys., 252
  (2013), pp.~495--517.

\end{thebibliography}

\end{document}